\documentclass[a4paper,10pt]{amsart}
\usepackage{txfonts,amsfonts}
\usepackage[arrow,matrix]{xy}
\usepackage{amsmath,amssymb,amscd,bbm,amsthm,mathrsfs,dsfont,bm}
\usepackage{extarrows}
\usepackage{latexsym}
\usepackage{graphicx}
\usepackage{color}
\usepackage{xy}
\usepackage{leftidx}
\input xy
\input xypic
\xyoption{all}

\allowdisplaybreaks
\newtheorem{maintheorem}{Theorem}

\newtheorem{theorem}{Theorem}[section]
\newtheorem{lemma}[theorem]{Lemma}
\newtheorem{proposition}[theorem]{Proposition}
\newtheorem{corollary}[theorem]{Corollary}
\newtheorem{hypothesis}[theorem]{Hypothesis}
\theoremstyle{definition}
\newtheorem{definition}[theorem]{Definition}

\newtheorem{remark}[theorem]{Remark}

\numberwithin{equation}{theorem}

\newtheorem{example}[theorem]{Example}
 
 \DeclareMathOperator{\Ext}{Ext}
 
\DeclareMathOperator{\Hom}{Hom}

\DeclareMathOperator{\uExt}{\underline{Ext}}

\DeclareMathOperator{\uHom}{\underline{Hom}}

\DeclareMathOperator{\gr}{gr}
\DeclareMathOperator{\id}{id}

\DeclareMathOperator{\GrMod}{Gr}
\DeclareMathOperator{\Gr}{Gr}
\DeclareMathOperator{\Ch}{Ch}

\newcommand{\AsmashB}{$A\otimes^{\tau}B$ }

\DeclareMathOperator{\hdet}{hdet}
\def\bt{\begin{theorem}}
\def\et{\end{theorem}}
\def\bl{\begin{lemma}}
\def\el{\end{lemma}}
\def\br{\begin{remark}}
\def\er{\end{remark}}
\def\bc{\begin{corollary}}
\def\ec{\end{corollary}}

\begin{document}
\title{Nakayama automorphisms of twisted tensor products
}

\author{Y. Shen}
\address{Shen: Department of Mathematics, Zhejiang Sci-Tech University, Hangzhou 310018, China}
\email{yuanshen@zstu.edu.cn}

\author{G.-S. Zhou}
\address{Zhou: Ningbo Institute of Technology, Zhejiang university, Ningbo 315100, China}
\email{10906045@zju.edu.cn}

\author{D.-M. Lu}
\address{Lu: School of Mathematical Sciences, Zhejiang University, Hangzhou 310027, China}
\email{dmlu@zju.edu.cn}
\date{}

\begin{abstract}
In this paper, we study homological properties of twisted tensor products of connected graded algebras. We focus on the Ext-algebras of twisted tensor products with a certain form of twisting maps firstly. We show those Ext-algebras are also twisted tensor products, and depict the twisting maps for such Ext-algebras in-depth. With those preparations, we describe Nakayama automorphisms of twisted tensor products of noetherian Artin-Schelter regular algebras.
\end{abstract}




\maketitle

\section*{Introduction}
The appearance of twisted tensor product of two algebras provides an impactful solution to the questions about ``factorization" and  ``product"  in quantum groups and noncommutative geometry (\cite{CIMZ,CSV,VV}). Compared with the common tensor product, twisted tensor product makes up for the limitation of too much commutativity, and is a suitable choice for constructing larger classes of noncommutative algebras. Roughly speaking, twisted tensor product $A\otimes^{\tau}\!B$ of algebras $A$ and $B$ is an associative algebra, which is $A\otimes B$ as vector spaces with the multiplication defined by a twisting map $\tau:B\otimes A\to A\otimes B$. The notion also has been called $R$-smash product in several recent papers (\cite{CIMZ,SWZ,WSL}).

For our research interests, $A$ and $B$ are always finitely generated algebras. In order to obtain nice combinatorial and homological properties, it is reasonable to demand the twisting map $\tau$ being determined by the action on generating set of $B\otimes A$. This requirement can be realized by an algebra homomorphism $\sigma=(\sigma_{ij})$ from $A$ to matrix algebra $\mathbb{M}_m(A)$ and a $\sigma$-derivation $\delta=(\delta_1,\cdots,\delta_m)^T$ from $A$ to $A^{\oplus m}$, where $m$ is the number of elements in the minimal generating set of $B$. In this case, we write $\tau=(\sigma,\delta)$. Significantly, classical Ore extensions and double Ore extensions with zero tails (see \cite{ZZ1} for detail) are both special cases of such twisted tensor products.

Artin-Schelter regular (AS-regular, for short) algebras play an important role in noncommutative algebras (\cite{AS}). The classification and construction of AS-regular algebras have acquired plentiful achievements. Ore extensions and double Ore extensions are commonly used to construct AS-regular algebras. Wang and the first and third named authors lead twisted tensor products into construction of AS-regular algebras in \cite{WSL}. The result tells that the twisted tensor product of AS-regular algebras $A$ and $B$ preserves AS-regularity in case $\sigma$ is invertible (see Definition \ref{invertible}) and $B$ has a pure resolution, namely, each term of the minimal free resolution of the trivial module concentrates in one degree (\cite[Theorem 0.3]{WSL}). The typical examples of algebras with pure resolutions include ($d$-)Koszul algebras and piecewise Koszul algebras. The AS-regularity of twisted tensor products inspires us to study more homological properties of this construction.

Nakayama automorphism is one of important homological invariants for AS-regular algebras, equivalently connected graded skew Calabi-Yau algebras. From the view of skew Calabi-Yau categories (see \cite{RRZ2}), the Nakayama automorphisms correspond to Serre functors of categories. In other side, Nakayama automorphisms control Hopf actions (\cite{CWZ,LMZ}). As special cases, Nakayama automorphisms of Ore extensions and trimmed double Ore extensions of Koszul AS-regular algebras have been studied in \cite{LWW,ZVZ}, respectively. Our aim is to give a description of the Nakayama automorphism of twisted tensor product $A\otimes^{\tau}\!B$ of noetherian AS-regular algebras $A$ and $B$, where the graded twisting map $\tau=(\sigma,\delta):B\otimes A\to A\otimes B$. It is the initial purpose of this paper.

To achieve the goal, we start from the study of Ext-algebras of twisted tensor products. The Ext-algebra $E(A):=\uExt_A^*(k,k)$ is an incredibly powerful tool used in many areas of mathematics for a connected graded algebra $A$. We focus on the Yoneda product of it. Let $A$ and $B$ be two connected graded algebras with $B$ having a pure resolution, and let $A\otimes^{\tau}\!B$ be a twisted tensor product for some graded twisting map $\tau=(\sigma,0)$. We construct a minimal free resolution of trivial module of $A\otimes^{\tau}\!B$ (see Proposition \ref{minimal resolution of Ck}), by the complex obtained in \cite[Theorem 0.2]{WSL}. This free resolution helps us bridge the relationships among the Ext-algebras $E(A)$, $E(B)$ and $E(A\otimes^{\tau}B)$. As a generalization of \cite[Theorem 3.7]{SWZ}, we get a result about Ext-algebras of twisted tensor products.

\begin{maintheorem}\label{main theorem:ext algebras} Let $A$ and $B$ be two connected Ext-finite graded algebras. Let $\tau=(\sigma,0):B\otimes A\to A\otimes B$ be a graded linear map such that $A\otimes^{\tau}B$ is a twisted tensor product. Denote the Ext-algebras of $A,B$ and $A\otimes^{\tau}B$ by $E(A),E(B)$ and $E(A\otimes^{\tau}B)$, respectively. Suppose $B$ has a pure resolution, then as bigraded algebras
$$
E(A\otimes^{\tau}\!B)\cong E(B)\otimes^{\tau_E}E(A),
$$
for some bigraded linear map ${\tau_E}:E(A)\otimes E(B)\to　E(B)\otimes E(A)$.
\end{maintheorem}

From the theorem above, we know the Yoneda product of Ext-algebra $E(A\otimes^{\tau}\!B)$ depends on the bigraded linear map ${\tau_E}$, which is relevant to the graded twisting map $\tau=(\sigma,0)$. So the properties of graded linear map $\sigma$ from $A$ to a matrix algebra $\mathbb{M}_m(A)$ over $A$, where $m$ is the number of elements in the minimal generating set of $B$, are determining factors for the Yoneda product of Ext-algebra $E(A\otimes^{\tau}\!B)$.

When we restrict ourselves on AS-regular algebras, there are two important homological invariants arising from $\sigma$. The first one is homological determinant $\hdet\sigma$, that is an $m\times m$ matrix over based field $k$ coming from the linear transformation of the direct sum $E(A)^{\oplus m}$ induced by $\sigma.$  It is a generalization for homological determinants of algebra automorphisms defined in \cite{JZ}.

 Besides, there is an automorphism $\det\sigma$ of $A$ such that as $(A,A)$-bimodules
$$\uExt^i_{A\otimes^{\tau}B}(A,A\otimes^{\tau}B)\cong\left\{\begin{array}{ll} {^{\det\sigma} A(l)} & \text{if }i= d,  \\0& \text{if }i\neq d, \end{array}\right.$$
for some integers $d,l$. The automorphism $\det\sigma$ is called determinant of $\sigma$. It is a generalization of determinant defined for double Ore extensions in \cite{ZZ1}. If $B$ is a polynomial algebra, $\det\sigma$ coincides with the classical determinant of matrices (see Example \ref{Example:det sigma=classical determinant of matrices}). Such two invariants of $\sigma$ make us describe the Yoneda product of Ext-algebra $E(A\otimes^{\tau}B)$ further, and provide an approach to study Nakayama automorphisms of twisted tensor products.

By using the theory of $A_\infty$-algebras, we have known an algebra is AS-regular if and only if its Ext-algebra is a Frobenius algebra (see \cite{LPWZ4}). Frobenius algebras have classical Nakayama automorphisms. By a result of Reyes, Rogalski and Zhang, Nakayama automorphisms of noetherian AS-regular algebras are dual to classical Nakayama automorphisms of their Ext-algebras (see \cite{RRZ2}). Our question turns to compute the Nakayama automorphisms of Ext-algebras.

Let  $A\otimes^\tau\!B$ be a twisted tensor product of two connected graded algebras $A, B$, where the graded twisting map $\tau=(\sigma,\delta):B\otimes A\to A\otimes B$ with nonzero $\sigma$-derivation. Note that the correlative graded linear map $\bar{\tau}=(\sigma,0):B\otimes A\to A\otimes B$ is also a graded twisting map, that is, $A\otimes^{\bar{\tau}}B$ is a twisted tensor product. The observation about the Yoneda product from the homological determinants and determinants mentioned above helps us describe the bigraded linear map $\tau_E$ for $A\otimes^{\bar{\tau}}\!B$ occurred in Theorem \ref{main theorem:ext algebras}. So we can compute the Nakayama automorphism of Ext-algebra $E(A\otimes^{\bar{\tau}}\!B)$. In particular, $A\otimes^{\bar{\tau}}\!B$ is an associated graded algebra of $A\otimes^{\tau}B$ for a canonical filtration. Following the properties of filtered algebras, we realize our ultimate goal.

\begin{maintheorem}\label{main theorem: nakayama automorphism}
Let $A$ and $B$ be both noetherian AS-regular algebras generated in degree $1$, and $Y=\{y_i\}_{i=1}^m$ be a minimal generating set of $B$. Let $\tau=(\sigma,\delta):B\otimes A\to A\otimes B$ be a graded twisting map such that associated twisted tensor product $A\otimes^{\bar{\tau}}\!B$ is noetherian. Denote the Nakayama automorphisms of $A,B$ and $A\otimes^{\tau}B$ by $\mu_A,\mu_B$ and $\mu_{A\otimes^{\tau}B}$, respectively. Suppose $B$ has a pure resolution and $\sigma$ is invertible, then Nakayama automorphism $\mu_{A\otimes^{\tau}B}$ satisfies
\begin{align*}
&{\mu_{A\otimes^{\tau}B}}_{|A}=(\det \sigma)^{-1}\circ\mu_A,\\
&{\mu_{A\otimes^{\tau}B}}
\left(
\begin{array}{c}
1\otimes^{\tau}y_1\\
1\otimes^{\tau}y_2\\
\vdots\\
1\otimes^{\tau}y_m
\end{array}
\right)
=
\hdet\sigma\cdot(A\otimes^{\tau}\,\mu_B)\left(
\begin{array}{c}
1\otimes^{\tau}y_1\\
1\otimes^{\tau}y_2\\
\vdots\\
1\otimes^{\tau}y_m
\end{array}
\right)
+
\left(
\begin{array}{c}
a_1\otimes^{\tau}1\\
a_2\otimes^{\tau}1\\
\vdots\\
a_m\otimes^{\tau}1
\end{array}
\right),
\end{align*}
where $a_1,a_2,\cdots,a_m\in A$ and $(A\otimes^{\tau}\,\mu_B)(1\otimes^{\tau} y_i)=1\otimes^{\tau}\mu_B(y_i)$ for $i=1,\cdots,m$.
\end{maintheorem}

This theorem coincides with the results obtained in \cite[Theorem 2]{LWW} and \cite[Theorem 2]{ZVZ} for the case of Ore extensions and trimmed double Ore extensions of Koszul AS-regular algebras.

The paper is organized as follows. In Section \ref{Section preliminaries}, we fix notations and   recall some definitions and results, as well as the definition of homological determinant is given.  In Section \ref{Section twisted tensor products}, we review some basic results of twisted tensor products, and discuss homological properties of twisted tensor products, including a minimal free resolution of trivial module of twisted tensor product, and the determinant for a graded algebra homomorphism from a connected graded algebra to a matrix algebra over it. In Section \ref{Section ext-algebras}, we focus on the Yoneda product of Ext-algebras, and Theorem \ref{main theorem:ext algebras} is proved.  In Section \ref{Section Nakayama automorphisms}, we prove Theorem \ref{main theorem: nakayama automorphism}. In Section \ref{Section Examples}, we give some applications of Theorem \ref{main theorem: nakayama automorphism}.

Throughout the paper, $k$ is a fixed field. All vector spaces and algebras are over $k$. Unless otherwise stated, tensor product $\otimes$ means $\otimes_k$.

\section{Preliminaries}\label{Section preliminaries}

\subsection{Notations and Definitions}
In this section, we fix basic notations and recall definitions for this paper.

Let $A=\bigoplus_{i\in\mathbb{Z}}A_i$ be a graded algebra. If $\dim A_i<\infty$ for all $i$, $A$ is called \emph{locally finite}. A locally finite graded algebra $A$ is called \emph{connected}, if $A_i=0$ for all $i<0$ and $A_0=k$. In this case, we write $\varepsilon_A:A\to k$ to be the augmentation for $A$.

We denote by $\GrMod A$ the category of left graded $A$-modules with morphisms consisting of left $A$-module homomorphisms preserving degrees. We identify $\GrMod A^o$ (resp. $\GrMod A^e$) with the category of graded right $A$-modules (resp. graded $A$-bimodules), where $A^o$ is the opposite algebra of $A$ and $A^e=A\otimes A^o$ is the enveloping algebra of $A$. Unless otherwise stated, all modules are left modules in this paper.

For each integer $i$, we define \emph{$i$-th shift} functor $-(i):\GrMod A\to \GrMod A$ by sending a graded module $M$ to $M(i)$, where $M(i)$ is a graded $A$-module with new grading $M(i)_j=M_{i+j}$ for each $j\in\mathbb{Z}$, and acting on morphisms as the identity map. Let $M,N\in \Gr A$ and $i$ be an integer, write graded vector spaces
$$
\uHom_A(M,N)=\bigoplus\nolimits_{j\in \mathbb{Z}}\Hom_{\Gr A}(M,N(j))\quad\text{and}\quad  \uExt^i_A(M,N)=\bigoplus\nolimits_{j\in \mathbb{Z}}\Ext^i_{\Gr A}(M,N(j)).
$$
Note that $\uHom_A (M,N)$ and $\uExt_A^{i}(M,N)$ become graded left modules if $M$ is a graded bimodule, graded right modules if $N$ is a  graded bimodule, and  graded bimodules if $M$ and $N$ are both graded bimodules. Similar remarks apply to the tensor product $L\otimes_AM$ of a graded right $A$-module $L$  ($L\otimes_AM$ is graded by $\deg\, (l\otimes m) = \deg l +\deg m$ when $l\in L$ and $m\in M$ are homogeneous).

Let $M$ be a graded $A$-bimodule, and let $\mu,\nu$ be graded algebra automorphisms of $A$. Then $^\mu M^\nu$ denotes the twisted $A$-bimodule such that $^\mu M^\nu=M$ as abelian groups, and where $a* m* a'=\mu(a)m\nu(a')$ for all $a,a'\in A$ and $m\in M$. If $\mu=\id$ (resp. $\nu=\id$), we use $M^\nu$ (resp. $^\mu M$) for $^{\id} M^\nu$ (resp. $^\mu M^{\id}$). Note that if $A$ is connected, it follows that $A^\mu\cong A^\nu$ as graded $A$-bimodules if and only if $\mu =\nu$.

\begin{definition}
A connected graded algebra $A$ is called an \emph{Artin-Schelter regular} (\emph{AS-regular}, for short) algebra of type $(d,l)$, if the following conditions hold:
\begin{enumerate}
\item[(AS1)]\label{AS1} $A$ has finite global dimension $d$.
\item[(AS2)]\label{AS2} $\uExt_{A}^{d}(k, A)=k(l)$ and $\uExt_{A}^{i}(k, A)=0$ for $i\neq d$.
\end{enumerate}

If $A$ is an AS-regular  algebra of type $(d,l)$, then by \cite[Proposition 4.5 (b)]{YZ}, there exists a unique graded automorphism $\mu$ of $A$ such that
$$\uExt^i_{A^e}(A,A^e)\cong\left\{\begin{array}{ll} 0&i\neq d,\\A^\mu(l)&i=d,   \end{array}\right. \quad \text{as graded $A$-bimodules}.$$
In this case, the automorphism $\mu$ is called the\emph{ Nakayama automorphism of $A$}, and always denoted by $\mu_A$.
\end{definition}

We denote by $\Ch(\Gr A)$ the category of cochain complexes of graded left $A$-modules with cochain complex morphisms.  For each integer $i$, the \emph{$i$-th shift} of cochain complex is a functor $-[i]:\Ch(\Gr A)\to \Ch(\Gr A)$ sending $X^\centerdot$ to a new cochain complex $X[i]$ with terms $X[i]^j=X^{j+i}$ and differential $(-1)^id_X$, and acting on morphisms as the identity map. For any $(X^\centerdot,d_X)\in \Ch(\Gr A)$ and $(Y^\centerdot,d_Y)\in \Ch(\Gr A^o)$, $Y^\centerdot\otimes_A X^\centerdot$ is a cochain complex where $n$-th term is $\bigoplus_{i+j=n}Y^i\otimes_A X^j$ and differential $d(y\otimes_Ax)=d_Y(y)\otimes_Ax+(-1)^iy\otimes d_X(x)$ for $x\in X^j$ and $y\in Y^i$.

When dealing with bigraded cases, we follow the Koszul sign rule for the first degree (such as homological degree) throughout the paper: when interchanging bihomogeneous elements of degree $(i_1, j_1)$ and $(i_2, j_2)$, multiply by $(-1)^{i_1i_2}$, which ignores the second degree. For example, a cochain complex $X^\centerdot$ in $\Ch (\Gr A)$ can be considered as a bigraded vector spaces $\bigoplus_i X^i$ with differential $d_X$ of degree $(1,0)$. Following the Koszul sign rule, the differential of $Y^\centerdot\otimes_A X^\centerdot$ has the form of $d_Y\otimes_A\id_{X} +\id_{Y}\otimes_Ad_X$.



\subsection{Ext-algebras}

Let $A$ be a connected graded algebra. There exists a graded free resolution $P^\centerdot$ of ${_Ak}$
\begin{equation*}
\xymatrix{
P^\centerdot=(\cdots\to A\otimes V_{3}\xlongrightarrow{d_P^{-3}} A\otimes V_{2}\xlongrightarrow{d_P^{-2}}A\otimes V_{1}
\xlongrightarrow{d_P^{-1}}A \to 0 \to\cdots ) \quad \xlongrightarrow{\varepsilon_A} \quad  {_Ak},
}
\end{equation*}
where $V_{i}$ is a graded vector space,
such that $\text{Im}\, d_P^{-i}\subseteq A_{\geq1}\otimes V_{i-1}$ for all $i=1,2,\cdots$. Here, we let $V_0=k$ and identify $A$ with $A\otimes k$. Such a graded free resolution is unique up to isomorphism in $\Ch(\Gr A)$, and is called the \emph{minimal free resolution} of ${_Ak}$.

Here and from now on, $(-)^*$ denotes the graded dual functor  of graded vector spaces. It is easy to see that the complex $\uHom_{A}(P^\centerdot,k)$ has zero differential. Hence the graded vector space $\uExt_A^i(k,k)=\uHom_A(A\otimes V_i, k) \cong V_i^*$ for all $i\geq 0$. We say $A$ is \emph{Ext-finite} if all $\uExt_A^i(k,k)$ are of finite dimension, or equivalently, all $V_i$  are finite dimensional, such as AS-regular algebras (\cite[Proposition 3.1]{SZ}). Also, we say $A$ has a \emph{pure resolution} if each $V_i$ concentrates in one degree. The examples of algebras with a pure resolution contain ($d$-)Koszul algebras and piecewise Koszul algebras.

%
%

The \emph{Ext-algebra} of $A$ is a bigraded vector space
$$
E(A):=\bigoplus_{(i,j)\in\mathbb{Z}^2}\uExt^i_A(k,k)_j
$$
equipped with the Yoneda product, where the second degree comes from the grading of $A$, written as subscript. The $(i,j)$-bihomogeneous piece of $E(A)$ is $E^i(A)_j=\Hom_{\Gr A}(P^{-i},k(j)).$

The paper \cite{SWZ} shows Ext-algebras for connected graded algebras give rise to a functor. We review it here for completion. Let $A,B$ be two connected graded algebras, and $f:A\to B$ be a graded algebra homomorphism. We define a bigraded algebra homomorphism $E(f):E(B)\to E(A)$ as follows.

There is a morphism of cochain complexes $\widetilde{f}:P^\centerdot\to Q^\centerdot$ in $\Ch(\Gr A)$ lifting $f$, such that $\varepsilon_B\circ f=\varepsilon_A$, where $P^\centerdot,Q^\centerdot$ are minimal free resolutions of $_Ak,{_Bk}$ respectively. Note that the graded algebra homomorphism $f$ makes all $B$-modules endow with $A$-module structures, then
$$
E(f):E(B)\cong\uHom_B(Q^{\centerdot},k)\hookrightarrow \uHom_A(Q^{\centerdot},k)\xlongrightarrow{\uHom_A(\widetilde{f},k)}\uHom_A(P^{\centerdot},k)\cong E(A).
$$
It can be checked $E(f)$ is an algebra homomorphism straightforwardly.
\begin{lemma}\label{Ext-algebra arise a functor} \cite[Theorem 2.4]{SWZ} Taking Ext-algebra
$E(-)$ is a contravariant functor from the category of connected graded algebras to the category of bigraded algebras.
%
\end{lemma}

\begin{definition}
Let $E=\bigoplus_{(i,j)\in\mathbb{Z}^2}E^i_{j}$ be a finite dimensional bigraded algebra with $E^0_0=k$. We say $E$ is a \emph{Frobenius algebra} of type  $(-d,l)$, if there is a nondegenerate associative bigraded bilinear form $\langle-,-\rangle:E\otimes E\to k$ of degree $(-d,l)$. By the nondegeneracy of the bilinear form, there exists a  bigraded automorphism $\mu$ of $E$ such that
$$
\langle a,b\rangle=(-1)^{i_1i_2}\langle \mu(b),a\rangle,
$$
for $a\in E^{i_1},b\in E^{i_2}$. It implies $E^*\cong\, ^\mu E(l)$ as bigraded $E^e$-modules. The unique automorphism $\mu$ is called the (classical) Nakayama automorphism, denoted by $\mu_E$. For more detail, see \cite{Sm}.
\end{definition}

The Ext-algebra $E(A)$ may be considered as a ``dual'' of $A$ in some sense, and recover lots of properties of $A$. We have following important results for AS-regularity and Nakayama automorphisms.
\begin{theorem}\label{the properties between ext algebras and AS regular algebras} Let $A$ be a connected graded algebra, and $E(A)$ be the Ext-algebra of $A$.
\begin{enumerate}
\item\cite[Corollary D]{LPWZ4} The algebra $A$ is AS-regular of type $(d,l)$ if and only if $E(A)$ is a Frobenius algebra of type $(-d,l)$.
\item \cite[Theorem 4.2]{RRZ2} If $A$ is noetherian AS-regular of type $(d,l)$ generated by degree $1$. Let $\mu_A$ and $\mu_E$ be Nakayama automorphisms of $A$ and $E(A)$, respectively. Then
    $$
    {\mu_E}_{|E^1(A)}=({\mu_A}_{|A_1})^*,
    $$
    where $E^1(A)$ is identified with $A_1^*$.
\end{enumerate}
\end{theorem}

\begin{remark}
The signs in the bigraded bilinear forms of Frobenius algebras come from the Koszul sign rule, which is different to the one defined in \cite{RRZ2}. So there is no additional sign in the formula for Nakayama automorphisms in Theorem \ref{the properties between ext algebras and AS regular algebras}(b), by a modification of signs to the proof of \cite[Theorem 4.2]{RRZ2}.
\end{remark}

\subsection{Homological determinant} Let $A$ be a connected graded algebra and $m$ be an arbitrary positive integer. Graded algebra homomorphisms from $A$ to matrix algebra $\mathbb{M}_m(A)$ play a key role in the discussion of twisted tensor products in the rest of paper. In this subsection, we introduce a definition of homological determinant for such graded algebra homomorphisms, which is a generalization of homological determinant defined in \cite{JZ,ZVZ}. Before that, we need some preparations.

\begin{definition}\label{phi-twisted module structure}
Let $M$ be a graded $A$-module and $U\cong k(l)^{\oplus m}$ be a graded vector space for some integer $l$. Let $\phi=(\phi_{ij}):A\to \mathbb{M}_m(A)$ be a graded algebra homomorphism, where each $\phi_{ij}:A\to A$ is a morphism of graded vector spaces. The \emph{$\phi$-twisted graded $A$-module} structure on  $M\otimes U $ is given by
$$
a *\left(\sum_{i=1}^m x_i\otimes e_i\right)=\sum_{j=1}^m\left(\sum_{i=1}^m \phi_{ji}(a)x_i\right)\otimes e_j,
$$
where $a\in A, x_i\in M$ and $\{e_i\}_{i=1}^m$ is the canonical coordinate basis of $U$. Such a graded $A$-module is denoted by ${^\phi(M\otimes U)}$. In particular, it is the usual twisted module structure when $m=1$.

Similarly, we can define $\phi$-twisted graded right $A$-modules.
\end{definition}

\begin{remark}
In fact, we can understand the $\phi$-twisted graded $A$-module structure by a multiplication of matrices. Note that $M\otimes U\cong M(l)^{\oplus m}$, so the $A$-action is
$$a* (x_1,x_2,\cdots,x_m)^T=\phi(a)\cdot(x_1,x_2,\cdots,x_m)^T$$
for $a\in A$ and $x_1,\cdots,x_m\in M$.
\end{remark}

\begin{definition}\label{invertible} Let $A$ be a connected graded algebra. A graded algebra homomorphism $\sigma=(\sigma_{ij}):A\to \mathbb{M}_{m}(A)$ is called \emph{invertible} if there is a graded algebra homomorphism $\varphi=(\varphi_{ij}):A\to \mathbb{M}_{m}(A)$
satisfying the following conditions:
$$
\sum_{k=1}^{m}\varphi_{jk} \circ \sigma_{ik}=
\left\{
\begin{array}{ll}
\id_A,&\text{if } \ i=j,\\
0,&\text{if }\ i\neq j,
\end{array}
\right.
\ \text{and}\ \
\sum_{k=1}^{m}\sigma_{kj}\circ \varphi_{ki}=
\left\{
\begin{array}{ll}
\id_A,&\text{if } \ i=j,\\
0,&\text{if }\ i\neq j.
\end{array}
\right.
$$
In this case, such $\varphi$ is unique, denoted by $\sigma^{-1}$.
\end{definition}

Now assume $A$ is AS-regular of type $(h,l)$. Let $\sigma:A\to \mathbb{M}_m(A)$ be a graded invertible algebra homomorphism. Write $\varphi=\sigma^{-1}$.

Choose a minimal free resolution of trivial module $_Ak$, that is, $P^\centerdot\xlongrightarrow{\varepsilon_A}{_Ak}$. Let graded vector space $U= k^{\oplus m}$. One obtains two quasi-isomorphisms
$$
\varepsilon_A^{\oplus m}:P^\centerdot \otimes U\xlongrightarrow{\varepsilon_A\otimes U} {_Ak}\otimes U\cong {_Ak}^{\oplus m},\quad\text{ and }\quad
\varepsilon_A^{\oplus m}:{^\varphi(P^\centerdot \otimes U)}\xlongrightarrow{\varepsilon_A\otimes U} {_Ak}\otimes U\cong {_Ak}^{\oplus m},
$$
where $i$-th term of ${^\varphi(P^\centerdot \otimes U)}$ is the $\varphi$-twisted graded $A$-module ${^\varphi(P^{-i} \otimes U)}$, since $^\varphi(k^{\oplus m})\cong k^{\oplus m}$ as graded $A$-modules.

It is not hard to check that each $^\varphi(P^{-i}\otimes U)$ is still a free $A$-module for $0\leq i\leq h_A$. Hence,  $P^\centerdot\otimes U$ and ${^\varphi (P^\centerdot\otimes U)}$ are both minimal free resolutions of $_Ak^{\oplus m}$, and the direct sum of $m$ copies of Ext-algebra of $A$
$$E(A)^{\oplus m}\cong \uHom_A({P^{\centerdot}\otimes U},k)\cong\uHom_A\left({^\varphi(P^\centerdot\otimes U)},k\right).$$

By Comparison Theorem, there exists an isomorphism of cochain complexes $\theta:{^\varphi(P^\centerdot\otimes U)}\to P^\centerdot\otimes U$ in $\Ch(\Gr A)$ such that $\varepsilon_A^{\oplus m}\circ \theta=\varepsilon_A^{\oplus m}$. Then we have a bigraded map
$$
\uHom_A(\theta,k):E(A)^{\oplus m}\to E(A)^{\oplus m}.
$$
In particular, $E^{h}(A)$ is one dimensional because of AS-regularity. Let $\omega^*$ be a basis of $E^{h}(A)$, and write $\{\omega_i^*=(0,\cdots,\omega^*,\cdots,0)\}_{i=1}^m$ to be the canonical basis set of $(E^{h}(A))^{\oplus m}$. When restricted on $(E^{h}(A))^{\oplus m}$, $\uHom_A(\theta,k)$ provides a matrix $H\in \mathbb{M}_m(k)$ satisfying
\begin{equation*}\label{hdet coefficient}
\uHom_A(\theta,k)(\omega^*_1,\cdots,\omega^*_m)^T=H(\omega^*_1,\cdots,\omega^*_m)^T.
\end{equation*}
We say the matrix $H$ is the \emph{homological determinant} of $\sigma$, denoted by $\hdet\sigma$.


\begin{remark}
If $m=1$, the definition of homological determinant here is coincident with the one given in \cite{JZ}. If $m=2$ and $A$ is Koszul algebra, it is also the same to the one in \cite{ZVZ}.
\end{remark}

Note that $\theta$ is an isomorphism of cochain complexes, we have an immediate result.
\begin{proposition}\label{hdet is invertible}
The homological determinant $\hdet \sigma$ is an invertible matrix.
\end{proposition}

\section{Twisted tensor products}\label{Section twisted tensor products}
In this section, we review definitions and some basic facts mainly from \cite{CIMZ,CSV,WSL}, and discuss some homological properties of twisted tensor products. We always restrict ourselves on (bi-)graded cases.

\subsection{Basic facts}Let $A$ and $B$ be two connected graded algebras, and $\tau: B\otimes A\to A\otimes B$ be a graded linear map. Define \AsmashB equals $A\otimes B$ as graded vector spaces with the multiplication given by
$$
m_{A\otimes^\tau B}=(m_{A}\otimes m_{B})(\id_{A}\otimes \tau\otimes \id_{B}),
$$
where $m_A,m_B$ are products of $A,B$ respectively. The elements can be presented by $a\otimes^{\tau}b$ in \AsmashB for any $a\in A,b\in B$. In general, \AsmashB is not an associative graded algebra.
\begin{definition}
We say $A\otimes^{\tau}\!B$ is a \emph{twisted tensor product}, if it is associative and $1_A\otimes^{\tau}1_B$ is the unit. In this case, $\tau$ is called a \emph{graded twisting map}.
\end{definition}

We adopt the same notation for twisted tensor products to the one used in \cite{WW}, by writing twisting maps as superscripts to avoid some confusion with the common tensor products over general rings. The usual tensor product is a special kind of twisted tensor products when the graded twisting map $\tau$ is chosen to be the filliping map, and we still use the notation $\otimes$ without superscript for it. The name ``$R$-smash product" is used in \cite{CIMZ,SWZ,WSL}, which is the same to twisted tensor product, we prefer the later one here.


To guarantee the unit condition, we need a \emph{normal} assumption on $\tau$: for all $a\in A,\; b\in B$,
$$\tau(b\otimes 1_A)=1_A\otimes b\quad  \text{and}\quad  \tau(1_B\otimes a)=a\otimes 1_B.$$   There is a necessary and sufficient condition for twisted tensor products (\cite[Theorem 2.5]{CIMZ}, \cite[Remark 2.4]{CSV}): $\tau$ is a graded twisting map if and only if $\tau$ is normal and quasitriangular, that is,
\begin{eqnarray*}
&&(\id_A\otimes\, m_B)(\tau\otimes\id_B)(\id_B\otimes \tau)=\tau\, (m_B\otimes \id_A),\\
&&(m_A\otimes \id_B)(\id_A\otimes \tau)(\tau\otimes \id_A)=\tau\,(\id_B\otimes\, m_A).
\end{eqnarray*}

In fact, there is a functor $A\otimes^{\tau}\!-: \GrMod B \to \GrMod\,(A\otimes^{\tau}\!B)$ defined as follows. For any graded module $M\in \GrMod B$,  $A \otimes^{\tau}\!M$ is the graded vector space  $A\otimes M$ equipped with the scalar multiplication given by $$(m_A\otimes \lambda_M) \circ (\id_A\otimes \tau\otimes \id_M),$$ where $m_A$ is the multiplication of $A$ and $\lambda_M$ is the scalar multiplication of $M$ as a graded $B$-module. The action of $A\otimes^{\tau}\!-$ on morphisms is the same to the map $A\otimes -$ on morphisms.

Similarly, we can define a functor $-\otimes^{\tau}\!B:\GrMod A^o\to \GrMod\,(A\otimes^{\tau}\!B)^o$.
\begin{lemma}\label{Atwtensor isomorphic to AtwtensorB tensor}
There are two pairs of natural isomorphic functors
$$
\begin{array}{llll}
A\otimes^{\tau}\!-\cong (A\otimes^{\tau}B)\!\otimes_B-:&\GrMod B&\to &\GrMod\,(A\otimes^{\tau}\!B),\\
-\otimes^{\tau}\!B\cong-\otimes_A(A\otimes^{\tau}\!B):&\GrMod A^o&\to&\GrMod\,(A\otimes^{\tau}\!B)^o.
\end{array}
$$
\end{lemma}

In order to obtain some nice combinatorial and homological properties for the twisted tensor product $A\otimes^{\tau}\!B$, we always restrict ourselves to the following situation.

Let $A=k\langle x_1,\cdots,x_n\rangle/(G_A)$ and $B=k\langle y_1,\cdots,y_m\rangle/(G_B)$ be two connected graded algebras. Since $\tau$ is just a graded linear map, we require $\tau$ being determined by its action on the generating set of $B\otimes A$ as in \cite{WSL}. That is, we ask the linear map $\tau$ has form of
$$\tau(y_j\otimes x_i)=\sum\nolimits_{t=1}^m a_{ijt}\otimes y_t+ a_{ij}\otimes 1,$$
where $a_{ijt},a_{ij}\in A$ and $i=1,\cdots,n,j=1,\cdots,m$. Following the quasitriangular condition for graded linear map $\tau$, we obtain an algebra homomorphism $\sigma=(\sigma_{ij}): A\to \mathbb{M}_m(A)$ and a $\sigma$-derivation $\delta=(\delta_i):A\to A^{\oplus m}$ such that
\begin{equation*}
\tau(y_{i}\otimes a)= \sum\nolimits_{t=1}^m \sigma_{it}(a)\otimes y_{s}+\delta_{i}(a)\otimes 1,\quad \forall a\in A,i=1,\cdots,m.
\end{equation*}
In the sequel, we write such form $\tau=(\sigma,\delta)$.

\begin{example}
For Ore extensions $A[z;\sigma,\delta]$ or double Ore extensions $A_P[y_1,y_2;\sigma,\delta]$, they are all special cases of twisted tensor products if we define $\tau=(\sigma,\delta)$ (see Example \ref{Example:Nakayama automorphisms of Ore extensions} and Example \ref{Example:Nakayama automorphisms of double Ore extensions} for detail).
\end{example}

Moreover, if we assume the graded algebra homomorphism $\sigma:A\to \mathbb{M}_m(A)$ is invertible in the sense of Definition \ref{invertible}, there exists a graded linear  map $\tau^{-1}:A\otimes B\to B\otimes A$, which is the inverse of $\tau$, satisfying
$$
\tau^{-1}(a\otimes y_i)= \sum\nolimits_{s=1}^m y_{s}\otimes \varphi_{si}(a) +1\otimes \delta'_{i}(a),
$$
where $a\in A,\delta'_i=-\sum_{j=1}^m\delta_{j}\circ \varphi_{ji}$ and $\varphi=\sigma^{-1}$ for $  i=1,\cdots,m.$  Following the quasitriangular condition, one obtains $\tau^{-1}$ is also a graded twisting map, equivalently,  $B\otimes^{\tau^{-1}}\!\!A$ is a twisted tensor product.

\begin{lemma}\label{isomorphism of ARB and AR-1B}\cite[Proposition 2.12]{WSL}
As graded algebras, $A\otimes^{\tau}\!B\cong B\otimes^{\tau^{-1}}\!\!A$.
\end{lemma}

Note that the definition of twisted tensor products can be given for bigraded algebras similarly. We close this subsection by exhibit the factorization theorem for twisted tensor products of bigraded algebras.

\begin{theorem}\label{fractorization theorem for twisted tensor product}\cite[Theorem 2.10]{CIMZ}
Let $A,B$ and $C$ be bigraded algebras. Then the following two conditions are equivalent
\begin{enumerate}
\item As bigraded algebras, $C\cong A\otimes^{\tau}B$ for some bigraded twisting map $\tau:B\otimes  A\to A\otimes B$;
\item  There exist bigraded algebra homomorphisms $f_A:A\to C$ and $f_B:B\to C$ such that  $$m_C\circ(f_A\otimes f_B):A\otimes  B\to C$$
 is an isomorphism of bigraded vector spaces, where $m_C$ is the multiplication of $C$.
\end{enumerate}
In this case, the bigraded twisting map $\tau={\Big (}m_C\circ(f_A\otimes f_B){\Big )}^{-1}\circ m_C\circ(f_B\otimes f_A)$.
\end{theorem}

%

\subsection{Minimal free resolutions}\label{Subsection Minimal free resolutiosn} The main goal of this subsection is to construct a minimal free resolution for the trivial module of a twisted tensor product of two connected graded algebras. In order to obtain some interesting homological relations between $A$, $B$ and the twisted tensor product $A\otimes^{\tau}\!B$, we always make the discussion with the following hypotheses throughout this paper, unless otherwise stated.

\begin{hypothesis}\label{hypothesis on twisted tensor products}
\begin{enumerate}
\item[]
\item Let $A$ be a connected Ext-finite graded algebra, and $\{x_1,\cdots,x_n\}$ be a minimal generating set of $A$;
\item Let $B$ be a connected Ext-finite graded algebra admitting a pure resolution, and $\{y_1,\cdots,y_m\}$ be a minimal generating set of $B$;
\item Let $\tau=(\sigma,0):B\otimes A\to A\otimes B$ be a graded twisting map, where $\sigma:A\to\mathbb{M}_m(A)$ is a graded algebra morphism.
\end{enumerate}
\end{hypothesis}

Under the hypotheses above, there are two graded algebra injections
\begin{equation*}
\begin{array}{ccclccccl}
\iota_A:&A&\to& A\otimes^{\tau}B&\qquad\qquad& \iota_B:&B&\to& A\ \otimes^\tau B\\
~&a&\mapsto& a\,\otimes^\tau1_B,&\qquad\qquad&~&b&\mapsto& 1_A\otimes^\tau b,
\end{array}
\end{equation*}
and two graded algebra surjections
\begin{equation*}
\begin{array}{ccccccccc}
\pi_A:&A\otimes^{\tau}B&\to & A&\qquad\qquad& \pi_B:&A\otimes^\tau B&\to & B\\
~&a\otimes^{\tau} b&\mapsto& a\cdot \varepsilon_B(b),&\qquad\qquad&~& a \otimes^{\tau}b&\mapsto& \varepsilon_A(a) \cdot b,
\end{array}
\end{equation*}
where $\varepsilon_A: A\to k$ and $\varepsilon_B:B\to k$ are the canonical augmentations. Clearly, $$\pi_A \circ \iota_A=\id_A \quad \text{ and } \quad \pi_B\circ \iota_B=\id_B.$$

By Lemma \ref{Ext-algebra arise a functor}, we have an immediate result.

\begin{corollary}\label{E(A) and E(B) are subalgebras}
The bigraded algebra homomorphisms $E(\pi_A):E(A)\to E(A\otimes^{\tau}\!B)$ and $E(\pi_B):E(B)\to E(A\otimes^{\tau}\!B)$ are both injective.
\end{corollary}
%
To show the bigraded algebra homomorphisms $E(\pi_A),E(\pi_B)$ and the Yoneda product of their images, we begin with minimal free resolutions of trivial modules. For convenience, write $C:=A\otimes^{\tau}B$.

\begin{remark}
The graded algebra homomorphisms $\pi_A$ and $\pi_B$ make all graded $A$-modules (resp. right $A$-modules) and graded $B$-modules (resp. right $B$-modules) be graded $C$-modules (resp. right $C$-modules). In most discussion of this paper, we will transfer objects and morphisms in $\Gr A$ and $\Gr B$ (resp. $\Ch(\Gr A)$ and  $\Ch(\Gr B)$) to ones in $\Gr C$ (resp. $\Ch (\Gr C)$) automatically.
\end{remark}

In the sequel, we fix a  minimal free resolution of $_Ak$ of the following form:
\begin{equation}\label{resolution of Ak}
P^\centerdot = (\cdots \xlongrightarrow{d_P^{-4}} A\otimes V_3 \xlongrightarrow{d_P^{-3}} A\otimes V_2 \xlongrightarrow{d_P^{-2}} A\otimes V_1 \xlongrightarrow{d_P^{-1}} A \to 0\to \cdots) \quad \xlongrightarrow{\varepsilon_A} \quad {_Ak} ,\tag{RA}
\end{equation}
where $V_i$ is a graded vector space for $i\geq 1$. Let $V_0= k$ and $V_i=0$ for $i<0$.

Since $B $ has a pure resolution, we may fix  a minimal free resolution of $_Bk$ of the form:
\begin{equation}\label{resolution of Bk}
Q^\centerdot = (\cdots \xlongrightarrow{d_Q^{-4}} B\otimes W_3\xlongrightarrow{d_Q^{-3}} B\otimes W_2 \xlongrightarrow{d_Q^{-2}} B\otimes W_1 \xlongrightarrow{d_Q^{-1}} B\to 0\to\cdots) \quad \xlongrightarrow{\varepsilon_B}\quad {_Bk},\tag{RB}
\end{equation}
satisfying that each $W_i= k(-l_{i})^{\oplus s_i} \text{ with } l_i, s_{i}\geq 1$ and $s_1=m$,  and $d_Q^{-1}$ assigns $1\otimes e_{j}\in B\otimes W_1$ to the coset of $y_j$ in $B$ for $j=1,\cdots,m$, where $\{e_{j}\}_{j=1}^m$ is the canonical coordinate basis of $W_1$. Write $W_0= k$ and $W_i=0$ for $i<0$.
%
%

Since the functor $A\otimes^{\tau}\!-$ is exact and preserves projective graded modules by Lemma \ref{Atwtensor isomorphic to AtwtensorB tensor}, one obtains  a minimal free resolution $F^\centerdot=A\otimes^{\tau} Q^\centerdot$  of $A$ in $\GrMod C$ of the form:
\begin{equation}
F^\centerdot = (\cdots \xlongrightarrow{d_F^{-4}} C\otimes W_3 \xlongrightarrow{d_F^{-3}} C\otimes W_2 \xlongrightarrow{d_F^{-2}} C\otimes W_1 \xlongrightarrow{d_F^{-1}} C\to 0\to\cdots) \quad \xlongrightarrow{\pi_A}\quad {}_CA.\tag{RC1}
\end{equation}
It is worthwhile to notice that the differential $d^{-i}_F=A\otimes^{\tau}d^{-i}_Q$ for each $i\geq0$.

The next key step is to make the minimal free resolution $F^\centerdot$ of ${}_CA$ be quasi-isomorphic to ${}_CA_A$ in $\Ch(\GrMod\, (C\otimes A^o))$. To this end, we use the assumption on $B$, that it has a pure resolution, to endow each term of resolution $F^\centerdot$ with a graded $\phi$-twisted right $A$-module structure for some graded algebra homomorphism $\phi$ from $A$ to a matrix algebra over $A$.

\begin{theorem}\label{resolution of A with bimodule structure}\cite[Theorem 3.7]{WSL}
There exists a sequence of graded algebra homomorphisms $\phi_n:A\to \mathbb{M}_{s_n}(A)$ for $n\geq2$, such that $\pi_A:F^\centerdot_\phi\to A$ is a quasi-isomorphism in $\Ch(\Gr\, (C\otimes A^o))$, where
\begin{equation}\label{resolution of BAA}
F_\phi^\centerdot = \cdots \xlongrightarrow{d_F^{-4}} (C\otimes W_3)^{\phi_3} \xlongrightarrow{d_F^{-3}} (C\otimes W_2)^{\phi_2} \xlongrightarrow{d_F^{-2}} (C\otimes W_1)^\sigma \xlongrightarrow{d_F^{-1}} C\to 0\to\cdots.\tag{RC2}
\end{equation}
Moreover, if $B$ is AS-regular algebra of type $(h_B,l_B)$ and $\sigma$ is invertible, then $F^\centerdot_\phi$ terminates at $(-h_B)$-th term, and $\phi_{h_B}$ is a graded automorphism of $A$.
\end{theorem}

\begin{proof}[Sketch of proof]
We describe the construction of complex $F_\phi^\centerdot$ briefly, which may be useful in Example \ref{Example:det sigma=classical determinant of matrices} and Example \ref{Example:twisted tensor product of special down-up algebras}, and see \cite{WSL} for the completed proof in a more general situation.

Firstly, it is obvious that the differential $d^{-1}_F$ becomes a graded right $A$-module homomorphism when the $(-1)$-th term of $F^\centerdot$ is endowed with a $\sigma$-twisted graded right $A$-module structure, because of the choice of $d^{-1}_F$. Then we obtain the sequence of graded algebra homomorphisms $\phi_n$ for $n\geq2$ inductively. Write $\phi_1=\sigma$ and $\phi_0=\id_A$.

Suppose we have got the desired graded algebra homomorphisms $\{\phi_{j}\}_{j=1}^{i-1}$ where $i\geq2$. For $\phi_{i}:A\to \mathbb{M}_{s_{i}}(A)$, we consider the following exact complex
$$
C\otimes W_{i}\xlongrightarrow{d^{-i}_F} (C\otimes W_{i-1})^{\phi_{i-1}} \xlongrightarrow{d^{-i+1}_F} (C\otimes W_{i-2})^{\phi_{i-2}}.
$$
Write the canonical coordinate basis of $W_{i}=k(-l_{i})^{\oplus s_{i}}$ to be $\{e_{j}^{i}\}_{j=1}^{s_{i}}$. For any $a\in A$ and $j=1,\cdots,s_{i}$, the element $(d^{-i}_F(1\otimes e_j^{i}))*a$ lies in the image of $d^{-i}_F$, since $d^{-i+1}_F$ is a graded right $A$-module homomorphism, say
$$
d^{-i}_F\left(\sum_{p=1}^{s_{i}}(c_{jp}\otimes e_p^{i})\right)=(d^{-i}_F(1\otimes e_j^{i}))*a,
$$
for some $c_{jp}\in C$ and $p=1,\cdots,s_{i}$. Note that $d^{-i}_F=A\otimes^{\tau} d^{-i}_Q$, and $W_{i}$ and $W_{i-1}$ are both concentrated in one degree. By an argument about degrees of elements in both sides of the equation above, we have that each $c_{jp}$ has form $a_{jp}\otimes^\tau1_B$, where $a_{jp}\in A$ for $p=1,\cdots,s_{i}.$ Define $\phi_i=((\phi_i)_{jp}):A\to \mathbb{M}_{s_i}(A)$ such that
$$
(\phi_{i})_{jp}(a)=a_{jp},
$$
for $j,p=1,\cdots,s_{i}.$ It can be checked $\phi_i$ is a well-defined graded algebra homomorphism, and make the differential
$d^{-i}_F: (C\otimes W_{i})^{\phi_i}\to(C\otimes W_{i-1})^{\phi_{i-1}}$ be a morphism in $\Gr\, (C\otimes A^o)$.
\end{proof}

\begin{proposition}\label{minimal resolution of Ck}
The morphism of cochain complexes  $\varepsilon_C:F_\phi^\centerdot\otimes_AP^\centerdot\to {_Ck}$ is a quasi-isomorphism. As a consequence, $F_\phi^\centerdot\otimes_AP^\centerdot$ is a minimal free resolution of $_Ck$.
\end{proposition}
\begin{proof}
The augmentation $\varepsilon_C$ of $C$ equals the composition of complexes morphisms
$$
F^\centerdot_\phi\otimes_AP^\centerdot\xlongrightarrow[\simeq]{\pi_A\otimes_AP^\centerdot}(_CA_A)\otimes_AP^\centerdot
\xlongrightarrow[\simeq]{A\otimes_A\varepsilon_A}(_CA_A)\otimes_A(_Ak)\cong {_Ck},
$$
and it is a quasi-isomorphism. The minimality is clear.
\end{proof}

The complex $F^\centerdot_\phi$ also helps us acquire the following result about AS-regularity of twisted tensor products.
\begin{theorem}\label{AS-regularity of twisted tensor products}\cite[Theorem 0.3]{WSL}
Let $A$ be AS-regular of type $(h_A,l_A)$ and $B$ be AS-regular of type $(h_B,l_B)$. Let $\tau=(\sigma,\delta):B\otimes A\to A\otimes B$ be a graded twisting map. If $B$ has a pure resolution and $\sigma$ is invertible, then $A\otimes^{\tau}\!B$ is AS-regular of type $(h_A+h_B,l_A+l_B)$.
\end{theorem}

\subsection{Determinant}For a graded algebra homomorphism $\sigma=(\sigma_{ij}):A\to \mathbb{M}_m(A)$, $\sigma$ is an $m\times m$ matrix consisting of graded linear maps $\{\sigma_{ij}:A\to A\}_{i,j=1}^m$. Classical determinant is an important invariant for a matrix, and we will give a generalized definition of determinant for such graded invertible algebra homomorphisms in the sense of Definition \ref{invertible}.

\begin{definition}\label{determinant--definition}
Let $A,B$ be two connected graded algebras, and $\sigma:A\to \mathbb{M}_m(A)$ be a graded invertible algebra homomorphism where $m$ is the number of elements in the minimal generating set of $B$. If graded linear map $\tau=(\sigma,0):B\otimes A\to A\otimes B$ is a graded twisting map and there exists a graded algebra automorphism $\nu$ of $A$ satisfying
$$
\uExt_{A\otimes^{\tau}B}^i(A,A\otimes^{\tau}\!B)\cong
\left\{
\begin{array}{ll}
0       &      i\neq d,  \\
{^{\nu}A}(l) &      i=d,
\end{array}
\right.
$$
as graded $A^e$-modules for some integers $d$ and $l$, we say $\nu$ is the \emph{determinant} of $\sigma$, denoted by $\det\sigma$.
\end{definition}

Since $A$ is connected, the determinant $\det \sigma$ is unique whenever it exists. The definition of determinant above is a generalization of the one given for the case of double Ore extensions by \cite[Theorem 0.4]{ZZ1}. However, the determinant does not exist in general. In the following, we focus on the interested case in which the determinant always exists.

We assume connected graded algebras $A,B$ and the graded linear map $\tau=(\sigma,0):B\otimes A\to A\otimes B$ satisfy Hypothesis \ref{hypothesis on twisted tensor products}. So $\sigma$ is a graded algebra homomorphism from $A$ to $\mathbb{M}_m(A)$. In addition, assume $B$ is AS-regular of type $(h_B,l_B)$ and $\sigma$ is invertible. Now the complex (\ref{resolution of BAA}) in Theorem \ref{resolution of A with bimodule structure} becomes
\begin{equation}\label{resolution of CAA--ASreg case}
F_\phi^\centerdot = \cdots\to 0 \to(C\otimes W_{h_B})^{\phi_{h_B}} \xlongrightarrow{d_F^{-{h_B}}} (C\otimes W_{h_B-1})^{\phi_{h_B-1}}\xlongrightarrow{d_F^{-h_{B}+1}}\cdots \xlongrightarrow{d_F^{-2}} (C\otimes W_1)^\sigma \xlongrightarrow{d_F^{-1}} C\to 0\to\cdots,\tag{RC3}
\end{equation}
where $W_{h_B}= k(-l_B)$ and $\phi_{h_B}$ is a graded automorphism of $A$.

The determinant of $\sigma$ is related to the cohomology of cochain complex $\uHom_C(F^\centerdot_\phi,C)$. Before computing the cohomologies, we transfer the Hom-set to the symmetrical version $B\otimes^{\tau^{-1}}\!\!A$.

Let $\psi:A\to\mathbb{M}_p(A)$ be a graded algebra homomorphism and $W= k(t)^{\oplus p}$ be a graded vector space for some integers $t$ and $p$. It is obvious,
$\uHom_{A\otimes^{\tau}B}\left((A\otimes^{\tau}B)\otimes W,A\otimes^{\tau}B\right)\cong W^*\otimes \left(A\otimes^\tau\! B\right)$ as graded $A^e$-modules. With additional $\psi$-twisted graded $A$-module structures, we have following isomorphisms in $\Gr A^e$,
\begin{eqnarray*}
\uHom_{A\otimes^{\tau}B}\left(\left((A\otimes^{\tau}B)\otimes W\right)^\psi,A\otimes^{\tau}B\right)
&\cong& {^\psi\left(W^*\otimes \left(A\otimes^\tau\! B\right)\right)}\\
&\cong& {^\psi\left(W^*\otimes\left(B\otimes^{\tau^{-1}}\!\!A\right)\right)}\\
&\cong& {^\psi\left(\left(W^*\otimes B\right)\otimes^{\tau^{-1}}\!\!A\right)}\\
&\cong& {^\psi\left(\uHom_{B}\left(B\otimes W,B\right)\otimes^{\tau^{-1}}\!\!A\right)}
\end{eqnarray*}
where the second isomorphism comes from Lemma \ref{isomorphism of ARB and AR-1B}, the third isomorphism follows from Lemma \ref{Atwtensor isomorphic to AtwtensorB tensor}, and the $\psi$-twisted graded $A$-module structure of ${^\psi\left(\uHom_{B}\left(B\otimes W,B\right)\otimes^{\tau^{-1}}\!\!A\right)}$ is induced by the one of $^\psi\left(\left(B\otimes^{\tau^{-1}}\!\!A\right)\otimes W^*\right)$. Now we have an isomorphism as below.
\begin{lemma}\label{isomorphism for homomorphism with twsited structures}
As graded $A^e$-modules,  $\uHom_{A\otimes^{\tau}B}\left(\left((A\otimes^{\tau}B)\otimes W\right)^\psi,A\otimes^{\tau}B\right)\cong {^\psi\left(\uHom_{B}\left(B\otimes W,B\right)\otimes^{\tau^{-1}}\!\!A\right)}$.
\end{lemma}

%
Then we can prove the main result of this subsection.
\begin{lemma}\label{det sigma--AS-reg}
The determinant $\det\sigma=\phi_{h_B}$.
\end{lemma}
\begin{proof}
Applying the functor $\uHom_C(-,C)$ to the complex $F_\phi^\centerdot$ (\ref{resolution of CAA--ASreg case}), we have a commutative diagram by Lemma \ref{isomorphism for homomorphism with twsited structures}
$$
\xymatrix{
\uHom_C((C\otimes W_{h_B})^{\phi_{h_B}},C)\ar[d]^{\cong}&\cdots\ar[l]&\uHom_C((C\otimes W_1)^\sigma,C)\ar[l]\ar[d]^{\cong}&\uHom_C(C,C)\ar[l]\ar[d]^{\cong}\\
{^{\phi_{h_B}}\left(\uHom_B(B\otimes W_{h_B},B)\otimes^{\tau^{-1}}\!\!A\right)}&\cdots\ar[l]&{^{\sigma}\left(\uHom_B(B\otimes W_{1},B)\otimes^{\tau^{-1}}\!\!A\right)}\ar[l]&\uHom_B(B,B)\otimes^{\tau^{-1}}\!\!A,\ar[l]
}
$$
and the bottom cochain complex is just $\uHom_B(Q^\centerdot, B)\otimes^{\tau^{-1}}\!\!A$, where $Q^\centerdot$ is the minimal resolution of ${_Bk}$, when forgetting the graded left $A$-module structures. Since the functor $-\otimes^{\tau^{-1}}\!\!A$ is exact and $B$ is AS-regular, the result follows.
\end{proof}

The following is an easy example to explain why we use the word ``determinant''.
\begin{example}\label{Example:det sigma=classical determinant of matrices}
Let $A$ be a connected graded algebra and $B=k[y_1,\cdots,y_m]$ be a polynomial algebra with $\deg y_i=l$ where $l\geq 1$. Let $\sigma:A\to \mathbb{M}_m(A)$ be a graded invertible algebra homomorphism such that $A\otimes^{\tau}B$ is a twisted tensor product where $\tau=(\sigma,0)$. Polynomial algebra $B$ is Koszul AS-regular of type $(m,ml)$, so the determinant of $\sigma$ exists and needs to construct the complex $F_\phi^\centerdot$ by Lemma \ref{det sigma--AS-reg}. Firstly, one obtains a minimal free resolution of $_Bk$
$$
(\cdots\to0\to B\otimes W_{m}\xlongrightarrow{d^{-m}}  B\otimes W_{m-1}\xlongrightarrow{d^{-m+1}}\cdots\xlongrightarrow{d^{-2}} B\otimes W_{1}\xlongrightarrow{d^{-1}} B\to 0\to\cdots)\quad\xlongrightarrow{\varepsilon_B}\quad {_Bk},
$$
where each graded vector space $W_i= k^{\oplus t_{i}}(-il)$ and $t_i=\binom{m}{i}$ with a canonical basis  $\{e_{j_1j_2\cdots j_i}\,|\,(j_1,\cdots,j_i)\in J_i\}$, index set $J_i=\{(j_1,\cdots,j_i)\,|\, j_1,\cdots,j_i=1,\cdots, m \text{ and } j_1<j_2<\cdots<j_i\}$, and each differential $d^{-i}$ satisfies
$$
d^{-i}(1\otimes e_{j_1j_2\cdots j_i})=\sum_{u=1}^i (-1)^{(u-1)}y_{j_u}\otimes e_{j_1\cdots j_{u-1}j_{u+1} \cdots j_i}.
$$

To make the following steps clear, it is necessary to sort the elements in each index set $J_i$. For any $(j_1,\cdots,j_i)$ and $(j'_1,\cdots,j'_i)$ in $J_i$, we say $(j_1,\cdots,j_i)<(j'_1,\cdots,j'_i)$ if there exists a positive integer $t$ such that $j_p=j'_p$ for $p<t$ and $j_t>j'_t$. According to the ordering $<$, we define a sorting map $\chi_i:J_i\to \{1,\cdots,\binom{m}{i}\}$ by sending the largest element $(1,2,\cdots,i)$ to $1$, second largest element $(1,\cdots,i-1,i+1)$ to $2$, and so forth.

The required complex can be acquired by acting the functor $A\otimes^{\tau}-$ to the minimal resolution above
$$
0\to ((A\otimes^{\tau}B)\otimes W_{m})^{\psi_m}\xlongrightarrow{A\otimes^{\tau}d^{-m}}  ((A\otimes^{\tau}B)\otimes W_{m-1})^{\psi_{m-1}}\xlongrightarrow{A\otimes^{\tau}d^{-m+1}}\cdots\xlongrightarrow{A\otimes^{\tau}d^{-2}} ((A\otimes^{\tau}B)\otimes W_{1})^{\sigma}\xlongrightarrow{A\otimes^{\tau}d^{-1}} B\to 0,
$$
where $\psi_{i}=((\psi_{i})_{pq}):A\to\mathbb{M}_{t_i}(A)$ is a graded algebra homomorphism with entry $(\psi_{i})_{pq}$ being the classical determinant of the submatrix formed by the $\chi_{i}^{-1}(p)$-th rows and $\chi_{i}^{-1}(q)$-th columns of $\sigma=(\sigma_{ij})$. Using the map $\chi_i$ to sort the basis $\{e_{j_1j_2\cdots j_i}\}_{J_i}$ of $W_i$, then the graded $\psi_i$-twisted right $A$-module structure is clear for $i=2,\cdots,m$.

In particular, $\psi_m$ is just the classical determinant of the matrix $\sigma=(\sigma_{ij})$. In other side, $\psi_m=\det \sigma$ by Lemma \ref{det sigma--AS-reg}.  That is to say, the classical determinant of matrix coincides with the one defined above in this case.  Thus the determinant given in Definition \ref{determinant--definition} is considered as a generalization of classical determinants for matrices.
\end{example}

\section{Ext-algebras of twisted tensor products}\label{Section ext-algebras}
In this section, we focus on the Yoneda products of Ext-algebras of twisted tensor products and prove the main result Theorem \ref{main theorem:ext algebras}, which is a more general context of \cite[Theorem 3.7]{SWZ}. Let $A, B$ be connected graded algebras and graded linear map $\tau=(\sigma,0):B\otimes A\to A\otimes B$ satisfy Hypothesis \ref{hypothesis on twisted tensor products}. We adopt the notations used in Section \ref{Section twisted tensor products}. Write $C:=A\otimes^{\tau}B$.

The minimal resolutions of $_Ak$, $_Bk$ and $_CA_A$ are still denoted by $P^\centerdot$, $Q^\centerdot$ and $F^\centerdot_\phi$ respectively. By Proposition \ref{minimal resolution of Ck}, $F_\phi^\centerdot\otimes_AP^\centerdot$ is a minimal free resolution of $_Ck$. Starting from those minimal free resolutions, we have
\begin{align*}
&E(A)=\bigoplus\nolimits_{i\geq0}\uHom_{A}(P^{-i},k)\cong\bigoplus\nolimits_{i\geq0} V_i^*,\\
&E(B)=\bigoplus\nolimits_{i\geq0}\uHom_{A}(Q^{-i},k)\cong \bigoplus\nolimits_{i\geq0}W_i^*,\\
&E(C)=\bigoplus\nolimits_{i\geq0}\uHom_{A}\left(\left(F_\phi^\centerdot\otimes_AP^\centerdot\right)^{-i},k\right)\cong \bigoplus\nolimits_{i\geq0}\bigoplus\nolimits_{p+q=i}W_p^*\otimes V_q^*.
\end{align*}
In the sequel, we identify the Ext-algebras $E^i(A), E^i(B)$ and $E^i(C)$ with $V_i^*,W_i^*$ and $\bigoplus_{p+q=i}W_p^*\otimes V_q^*$ for each $i\geq0$ respectively.
\begin{remark}\label{a note for E(A), E(B) and E(C)}
We treat Ext-algebras $E(A)$, $E(B)$ and $E(C)$ as graded $k$-dual spaces of some graded vector spaces to make notations uniform. However,  we need to keep in mind that each element in such identifications for Ext-algebras has a corresponding representation through minimal free resolutions as in \cite[Remark 3.3]{SWZ}.

For any homogeneous element $f\in (V^*_i)_{-t}=E^i(A)_{-t}$, there is a corresponding morphism  $\alpha:=P^{\centerdot}\to k(-t)[i]$ in $\Ch(\Gr A)$ induced by $\varepsilon_A\otimes f: A\otimes V_i\to {_Ak}(-t)$. Similarly, to the homogeneous elements in $E^i(B)$.

In other side, for any $g\otimes f\in (W^*_{j})_{-s}\otimes (V^*_{i})_{-t}\subseteq E^{i+j}(C)_{-s-t}$, there is a corresponding morphism  $\gamma:=F^{\centerdot}_\phi\otimes_AP^\centerdot\to {_Ck}(-s-t)[i+j]$ in $\Ch(\Gr C)$ induced by
$$
(C\otimes W_j)^{\phi_j}\otimes_A(A\otimes V_{i})\xlongrightarrow{(C\otimes W_j)^{\phi_j}\otimes(A\otimes f)}(C\otimes W_j)(-t)\xlongrightarrow{(\pi_B\otimes W_j)(-t)} (B\otimes W_j)(-t)\xlongrightarrow{(\varepsilon_B\otimes g)(-t)} {_Ck}(-s-t).$$
Such representations are useful to prove results below.
\end{remark}

We are aimed at describing the connections between $E(A), E(B)$ and $E(C)$.  For this purpose, we follow the action on morphisms of the functor $E(-)$ to characterize injections $E(\pi_A)$ and $E(\pi_B)$ through the minimal free resolutions $P^\centerdot$, $Q^\centerdot$ and $F^\centerdot_\phi\otimes_AP^\centerdot$.

Since the algebra homomorphism $\pi_A:C\to A$ makes all $A$-modules be $C$-modules, so the cochain complex $P^\centerdot$ is also an object of $\Ch(\Gr C)$. And $\pi_A:F^\centerdot_\phi\to A$ is a morphism of cochain complexes in $\Ch(\Gr\,(C\otimes A^o))$, so we have a morphism in $\Ch(\Gr C)$:
$$
\widetilde{\pi}_A: F_\phi^\centerdot\otimes_A P^\centerdot\xlongrightarrow{\pi_A\otimes_AP^\centerdot} A\otimes_AP^\centerdot\cong P^\centerdot,
$$
lifting $\pi_A$ such that $\varepsilon_A\circ \widetilde{\pi}_A=\varepsilon_C$.

In other side, the composition $A\xlongrightarrow{\iota_A}C\xlongrightarrow{\pi_B}B$ of graded algebra homomorphisms endows each graded right $B$-module with a trivial graded right $A$-module structure. Then $(B\otimes W_i)^{\phi_{i}}\cong B\otimes W_i$ as graded $B\otimes A^o$-modules. One obtains a morphism of cochain complexes $\widetilde{\pi}'_B:F^\centerdot_\phi\to Q^\centerdot$ in $\Ch(\Gr\,(C\otimes A^o))$, of which each $i$-th component is $\pi_B\otimes W_i$, satisfying $\varepsilon_A\circ\pi_A=\varepsilon_B\circ\widetilde{\pi}'_B$. Now we have a composition of morphisms lifting $\pi_B$ in $\Ch(\Gr C)$:
$$
\widetilde{\pi}_B: F_\phi^\centerdot\otimes_A P^\centerdot\xlongrightarrow{\widetilde{\pi}_B'\otimes_AP^\centerdot} Q^\centerdot\otimes_AP^\centerdot\xlongrightarrow{\mathrm{pr}} Q^\centerdot\otimes_AA\cong Q^\centerdot,
$$
where $\mathrm{pr}:Q^\centerdot\otimes_AP^\centerdot\to Q^\centerdot\otimes_A A$ is a projection of cochain complexes by the minimality of $P^\centerdot$. Clearly, $\varepsilon_B\circ\widetilde{\pi_B}=\varepsilon_C$.


\begin{lemma}\label{Images of E(pi_A) and E(pi_B)}
For each $i\geq0$, let $f\in V_i^*$ and $g\in W^*_i$, then
$$
E(\pi_A)(f)=1\otimes f\in W_0^*\otimes V_i^*\subseteq E^i(C)\quad\text{and}\quad
E(\pi_B)(g)=g\otimes 1\in W_i^*\otimes V_0^*\subseteq E^i(C).
$$
\end{lemma}
\begin{proof}Without loss of generality, we assume $f\in (V_i^*)_{-t}$ is a homogeneous element.
Let $\alpha:P^\centerdot\to {_Ak}(-t)[i]$ be the corresponding morphism of cochain complexes to $f$ by Remark \ref{a note for E(A), E(B) and E(C)}, whose $(-i)$-th component is $\varepsilon_A\otimes f:A\otimes V_i\to {_Ak}(-t)$.

Following the definition of $E(\pi_A)$ above the Lemma \ref{Ext-algebra arise a functor}, we have $E(\pi_A)(\alpha)=\alpha\circ\widetilde{\pi}_A$. However, the nonzero values of $(-i)$-th component of $\widetilde{\pi}_A$ only take from the direct summand $C\otimes_A (A\otimes V_i)$ of the $(-i)$-th term of $F^\centerdot_\phi\otimes_AP^\centerdot$. Hence, the only nonzero part of $E(\pi_A)(\alpha)$ equals the composition of following morphisms
$$
C\otimes_A(A\otimes V_i)\xlongrightarrow{\pi_A\otimes_A (A\otimes V_i)} A\otimes_A(A\otimes V_i)\cong A\otimes V_i\xlongrightarrow{\varepsilon_A\otimes f}{_Ck}.
$$
Then $E(\pi_A)(f)=1\otimes f$ corresponds to $E(\pi_A)(\alpha)$ when identifying $E^i(C)=\bigoplus_{j=0}^iW^*_{j}\otimes V^*_{i-j}$ by Remark \ref{a note for E(A), E(B) and E(C)}.

Similarly, we have $E(\pi_B)(g)=g\otimes 1\in W_i^*\otimes V_0^*$ by the morphism $\widetilde{\pi}_B$ in $\Ch (\Gr C)$.
\end{proof}

Next step is to compute the Yoneda products of images of $E(\pi_A)$ and ones of $E(\pi_B)$ in $E(C)$. We divide it into three cases.

\subsection{General case} We prove the first main result of this paper here. Before that, we describe one case of Yoneda products in $E(C)$ under Hypothesis \ref{hypothesis on twisted tensor products} without additional conditions.

\begin{lemma}\label{The Yoneda product of E(pi_B) and E(pi_A)}
Let $f\in V_i^*, g\in W_j^*$, then
$$
(g\otimes 1)\cdot(1\otimes f)=(-1)^{ij}g\otimes f\in W_j^*\otimes V_i^*\subseteq E^{i+j}(C).
$$
\end{lemma}
\begin{proof}
Without loss of generality, we assume $f\in (V^*_i)_{-t}$  and $g \in (W^*_j)_{-l_j}$ are both homogeneous elements. Write $\alpha:P^\centerdot\to{_Ak}(-t)[i]$ and $\beta:Q^\centerdot\to {_Bk}(-l_j)[j]$ to be the corresponding morphisms of $f$ and $g$ respectively. By the proof of Lemma \ref{Images of E(pi_A) and E(pi_B)} and Remark \ref{a note for E(A), E(B) and E(C)}, $E(\pi_A)(\alpha):F^\centerdot_\phi\otimes_AP^\centerdot\to {_Ck}(-t)[i]$ in $\Ch(\Gr C)$ corresponds to $1\otimes f\in W_0^*\otimes V^*_i$ induced by
$$
\pi_A\otimes_A(\varepsilon_A\otimes f):C\otimes_A(A\otimes V_i) \to {_Ck(-t)},
$$
and $E(\pi_B)(\beta):F^\centerdot_\phi\otimes_AP^\centerdot\to {_Ck}(-l_j)[j]$ in $\Ch(\Gr C)$ corresponds to $g\otimes 1\in W^*_j\otimes V_0^*$ induced by
$$
(C\otimes W_j)^{\phi_j}\otimes_AA\cong C\otimes W_j\xlongrightarrow{\pi_B\otimes W_j}B\otimes W_j\xlongrightarrow{\varepsilon_B\otimes g} {_Ck}(-l_j).
$$

By Comparison Theorem, choose a morphism $\widetilde{\alpha}:P^\centerdot\to P^\centerdot(-t)[i]$ in $\Ch(\Gr A)$ such that $\varepsilon_A(-t)[i]\circ\widetilde{\alpha}=\alpha$. To be specific, the $(-i)$-th component of $\widetilde{\alpha}$ is $A\otimes f:A\otimes V_i\to A(-t)$.

Now consider the following commutative diagram in $\Ch(\Gr C)$:
$$
\xymatrix{
F^\centerdot_\phi\otimes_A (P^\centerdot(-t)[i])\ar[rd]^{\lambda}&F^\centerdot_\phi\otimes_A P^\centerdot
\ar[d]_{\psi}
\ar[rd]^{E(\pi_A)(\alpha)}
\ar[l]_(0.4){F^\centerdot_\phi\otimes_A\widetilde{\alpha}}
&
\\
&(F^\centerdot_\phi\otimes_A P^\centerdot)(-t)[i]
\ar[r]^(0.6){\varepsilon_C(-t)[i]}
\ar[d]_{E(\pi_B)(\beta)(-t)[i]}
&{_Ck}(-t)[i]\\
&_Ck(-t-l_j)[i+j],
}
$$
where $\lambda: F^\centerdot_\phi\otimes_A (P^\centerdot(-t)[i])\to (F^\centerdot_\phi\otimes_A P^\centerdot)(-t)[i]$ acts on each summand $F^{-j}_\phi\otimes_AP^{-u}$ by multiplying $(-1)^{ij}$ for $j,u\geq 0$ and $\psi=\lambda\circ (F^\centerdot_\phi\otimes_A\widetilde{\alpha})$. It is easy to check $\varepsilon_C(-t)[i]\circ \psi=E(\pi_A)(\alpha)$.

Then $E(\pi_B)(\beta)\cdot E(\pi_A)(\alpha)=E(\pi_B)(\beta)(-t)[i]\circ \psi,$ and the nonzero part equals composition of following morphisms
$$
(C\otimes W_j)^{\phi_j}\otimes_A (A\otimes V_i)\xlongrightarrow{(-1)^{ij}(C\otimes W_j)^{\phi_j}\otimes_A (A\otimes f)} (C\otimes W_j)^{\phi_j}(-t)\xlongrightarrow{(\pi_B\otimes W_j)(-t)} (B\otimes W_j)(-t)\xlongrightarrow{(\varepsilon_B\otimes g)(-t)} {_Ck}(-t-l_j).
$$
Hence, $(g\otimes 1)(1\otimes f)=(-1)^{ij}g\otimes f$ by Remark \ref{a note for E(A), E(B) and E(C)}.
\end{proof}

\begin{theorem}\label{E(C) is a twisted tensor product of E(B) and E(A)}
There is a bigraded twisting map ${\tau_E}:E(A)\otimes E(B)\to E(B)\otimes E(A)$ satisfying
$$
m_{E(C)}\circ(E(\pi_B)\otimes E(\pi_A))\circ {\tau_E}=m_{E(C)}\circ(E(\pi_A)\otimes E(\pi_B)),
$$
where $m_{E(C)}$ is the Yoneda product of $E(C)$, such that as bigraded algebras,
$$
E(C)\cong E(B)\otimes^{\tau_E}\!E(A).
$$
\end{theorem}
\begin{proof}
By Corollary \ref{E(A) and E(B) are subalgebras}, we know that $E(\pi_A):E(A)\to E(C)$ and $E(\pi_B):E(B)\to E(C)$ are bigraded algebra homomorphisms. Write bigraded linear map
$$
\rho:=m_{E(C)}\circ (E(\pi_B)\otimes E(\pi_A)):E(B)\otimes E(A)\to E(C).
$$
Following from Lemma \ref{Images of E(pi_A) and E(pi_B)} and Lemma \ref{The Yoneda product of E(pi_B) and E(pi_A)}, $\rho$ is an isomorphism of bigraded vector spaces.  By Theorem \ref{fractorization theorem for twisted tensor product}, as bigraded algebras
$$
E(C)\cong E(B)\otimes^{\tau_E}\!E(A),
$$
where the bigraded twisting map ${\tau_E}=\rho^{-1}\circ m_{E(C)}\circ (E(\pi_A)\otimes E(\pi_B)):E(A)\otimes E(B)\to E(B)\otimes E(A)$.
\end{proof}

A good will is to give a specific formula for the bigraded twisting map ${\tau_E}$ in theorem above as the \cite[Theorem 3.7]{SWZ}. Unfortunately, we cannot make our wish come true, since it is hard to know what $(1\otimes f)\cdot (g\otimes 1)$ is where $f\in E(A)$ and $g\in E(B)$ in general. However, we may give a partially depict for such Yoneda products with AS-regularity in the following two subsections.

\subsection{When $B$ is AS-regular}We assume $B$ is AS-regular of type $(h_B,l_B)$ and $A$ is generated in degree $1$ in addition. In this case, let $A_1=V_1$ which is the graded vector space in minimal resolution (\ref{resolution of Ak}) of $_Ak$.

The minimal free resolution (\ref{resolution of Bk}) of trivial module $_Bk$ becomes
$$
Q^\centerdot:=(\cdots\to0\to B\otimes W_{h_B}\xlongrightarrow{d^{-h_B}_{Q}} B\otimes W_{h_B-1}\xlongrightarrow{d^{-h_B+1}_Q}\cdots  \xlongrightarrow{} B\otimes W_1\xlongrightarrow{d^{-1}_Q} B\to0\to\cdots)\quad \xlongrightarrow{\varepsilon_B}\quad {_Bk},
$$
where $W_{h_B}=k(-l_B)$. Recall the resolution (\ref{resolution of CAA--ASreg case}) of ${_CA_A}$ by Lemma \ref{det sigma--AS-reg},
\begin{equation*}
F_\phi^\centerdot = (\cdots\to 0 \to(C\otimes W_{h_B})^{\det \sigma} \xlongrightarrow{d_F^{-{h_B}}} (C\otimes W_{h_B-1})^{\phi_{h_B-1}}\xlongrightarrow{d_F^{-h_{B}+1}}\cdots \xlongrightarrow{d_F^{-2}} (C\otimes W_1)^\sigma \xlongrightarrow{d_F^{-1}} C\to 0)\xlongrightarrow{\pi_A} {_CA_A}.
\end{equation*}
Obviously, there are two morphisms of cochain complexes:
$$
\qquad\iota_0:C\to F^\centerdot_\phi,\qquad\quad \pi_{h_B}:F^\centerdot_\phi\to (C\otimes W_{h_B})^{\det \sigma}[h_B].
$$
\begin{lemma}\label{E^1(A) *E^hB(B)}
Let $f\in V_1^*$ and $g\in W_{h_B}^*$. If $\sigma$ is invertible, then
$$
(1\otimes f)\cdot (g\otimes 1)= g\otimes (f\circ (\det\sigma)_{|\,V_1})\in W_{h_B}^*\otimes V_1^*\subseteq E^{h_B+1}(C).
$$
\end{lemma}
\begin{proof}
Write $\nu=\det\sigma$. Let $\alpha:P^\centerdot\to {_Ak(-1)}[1]$ be the morphism of cochain complexes corresponding to $f$ and $\beta:Q^\centerdot \to {_Bk(-l_B)}[h_B]$ be the one corresponding to $g$. By Lemma \ref{Images of E(pi_A) and E(pi_B)} and Remark \ref{a note for E(A), E(B) and E(C)}, $E(\pi_A)(\alpha):F^\centerdot_\phi\otimes_AP^\centerdot\to {_Ck}(-1)[1]$ in $\Ch(\Gr C)$ corresponds to $1\otimes f\in W_0^*\otimes V^*_{1}$ induced by
$$
\pi_A\otimes_A(\varepsilon_A\otimes f):C\otimes_A(A\otimes V_1) \to {_Ck(-1)},
$$
and $E(\pi_B)(\beta):F^\centerdot_\phi\otimes_AP^\centerdot\to {_Ck}(-l_B)[h_B]$ in $\Ch(\Gr C)$ corresponds to $g\otimes 1\in W^*_{h_B}\otimes V_0^*$ induced by
$$
(C\otimes W_{h_B})^{\nu}\otimes_AA\cong C\otimes W_{h_B} \xlongrightarrow{\pi_B\otimes W_{h_B}}B\otimes W_{h_B}\xlongrightarrow{\varepsilon_B\otimes g} {_Ck}(-l_B).
$$

Clearly, the cochain complex ${^{\nu^{-1}} P^\centerdot}$ is also a minimal free resolution of $_Ak$ via quasi-isomorphism $\varepsilon_A$. For each $i$-th term of ${^{\nu^{-1}} P^\centerdot}$, it is obvious ${^{\nu^{-1}} (A\otimes V_i)}\cong ({^{\nu^{-1}} A})\otimes V_i $ as graded $A$-modules. One can choose an isomorphism of cochain complexes $\rho:{^{\nu^{-1}} P^\centerdot}\to P^\centerdot$ in $\Ch(\Gr A)$ such that $\varepsilon_A\circ \rho=\varepsilon_A$ and
$$
\begin{array}{ll}
\rho^{0}=\nu:&{^{\nu^{-1}}A}\to A,\\
\rho^{-1}=\nu\otimes\nu_{|\,V_1}:& (^{\nu^{-1}}A)\otimes V_1\to A\otimes V_1.
\end{array}
$$

The Yoneda product can be computed by the following commutative diagram in $\Ch(\Gr C)$:
$${\small
\xymatrix{   \left((C\otimes W_{h_B})^\nu\otimes_A P^\centerdot\right)[h_B]\ar[d]_{\cong}^{\eta_2}&
  (C\otimes W_{h_B})^\nu[h_B]\otimes_A P^\centerdot\ar[l]^{\cong}_{\eta_1}&
          F^\centerdot_\phi\otimes_AP^\centerdot\ar[l]_(0.35){\pi_{h_B}\otimes_A P^\centerdot}\ar[dd]^{\psi}\ar[rrdd]^{E(\pi_B)(\beta)}\\
            \left((C\otimes W_{h_B})\otimes_A{^{\nu^{-1}} P^\centerdot}\right)[h_B]\ar[d]^{\eta_3}& \\
          \left((C\otimes W_{h_B})\otimes_A P^\centerdot\right)[h_B]\ar[r]^{\eta_4}   &
          (C\otimes_AP^\centerdot)(-l_B)[h_B]\ar[r]^{\eta_5}
          &(F^\centerdot_\phi\otimes_AP^\centerdot)(-l_B)[h_B]\ar[rr]^(0.55){\varepsilon_C(-l_B)[h_B]}\ar[d]^{E(\pi_A)(\alpha)(-l_B)[h_B]}& &{_C}k(-l_B)[h_B]\\
          &&{_Ck}(-l_B-1)[h_B+1],&
}}
$$
where $\psi:=\eta_5\circ \eta_4\circ \eta_3 \circ \eta_2\circ\eta_1\circ (\pi_{h_B}\otimes_A P^\centerdot)$, the isomorphism $\eta_1$ is a natural shift of complexes without additional signs, the isomorphism $\eta_2$ acts on elements as the identity map, $\eta_3:=\left((C\otimes W_{h_B})\otimes_A\rho\right)[h_B]$, $\eta_4:=\left((C\otimes g)\otimes_A P^\centerdot\right)[h_B]$, and $\eta_5:=(\iota\otimes_AP^\centerdot)(-l_B)[h_B]$. It is easy to check that $\varepsilon_C(-l_B)[h_B]\circ \psi=E(\pi_B)(\beta)$.

Then $E(\pi_A)(\alpha)\circ E(\pi_B)(\beta)=E(\pi_A)(\alpha)(-l_B)[h_B]\circ \psi$, and the nonzero component equals composition of following morphisms in $\Gr C$
$$
(C\otimes W_{h_B})^\nu \otimes_A (A\otimes V_1)\cong (C\otimes W_{h_B})\otimes_A\, {^{\nu^{-1}}\!(A\otimes V_1) }\xlongrightarrow{(C\otimes g)\otimes_A(\nu\otimes\nu)} C\otimes_A(A\otimes V_1)(-l_B)\xlongrightarrow{(\pi_A\otimes_A(\varepsilon_A\otimes f))(-l_B)}{_Ck}(-l_B-1).
$$
Hence $(1\otimes f)\cdot (g\otimes 1)=g\otimes (f\circ \nu)\in W_{h_B}^*\otimes V_1^*\subseteq E^{h_B+1}(C)$ by Remark \ref{a note for E(A), E(B) and E(C)}.
\end{proof}

\subsection{When $A$ is AS-regular} In the final case, we assume $A$ is AS-regular of type $(h_A,l_A)$, and $\sigma$ is invertible. Write $\varphi=\sigma^{-1}$ and $U=k^{\oplus m}$ in the sequel.

The minimal resolution (\ref{resolution of Ak}) of $_Ak$ becomes:
$$
P^\centerdot:=(\cdots\to0\to A\otimes V_{h_A}\xlongrightarrow{d^{-h_A}_{P}} A\otimes V_{h_A-1}\xlongrightarrow{d^{-h_A+1}_P}\cdots  \xlongrightarrow{} A\otimes V_1\xlongrightarrow{d^{-1}_P} A\to0\to\cdots)\quad \xlongrightarrow{\varepsilon_A}\quad {_Ak},
$$
where $V_{h_A}=k(-l_A)$. Let $\omega$ be a basis of $V_{h_A}$. So the dual basis $\omega^*$is a basis of $E^{h_A}(A)= V_{h_A}^*$.

By the definition, the homological determinant of $\sigma$ comes form the morphism $\theta:{^\varphi (P^\centerdot\otimes U)}\to P^\centerdot\otimes U$ satisfying $\varepsilon_A^{\oplus m}\circ \theta=\varepsilon_A^{\oplus m}$. In the following, we use other morphisms to obtain homological determinant.

Let $\mathrm{pr_i}:{_Ak^{\,\oplus m}}\to{_Ak}$ be the natural projection of $i$-th position for $1\leq i\leq m$. Then there exists a morphism of cochain complexes in $\Ch(\Gr A)$
$$
\theta_{i}:{^\varphi (P^\centerdot\otimes U)} \to P^\centerdot,
$$
such that $\varepsilon_A\circ\theta_i=\mathrm{pr}_i\circ\varepsilon_A^{\oplus m}$ for $1\leq i\leq m$.


Let $i$ be an arbitrary integer in the set $\{1,\cdots,m\}$. The action of functor $\uHom_A(-,k)$ to $\theta_i$ produces a bigraded linear map $\uHom_A(\theta_i,k):E(A)\to E(A)^{\oplus m}$. Restricted on $E^{h_A}(A)$, one obtains
$$
\uHom_A(\theta_i,k)(\omega^*)=h'_{i1}\omega^*_1+h'_{i2}\omega^*_2+\cdots+h'_{im}\omega^*_m,
$$
where $h'_{i1},h'_{i2},\cdots,h'_{im}\in k$ and $\{\omega^*_i\}$ is the canonical basis of $E^{h_A}(A)^{\oplus m}$ with respect to $\omega^*$. Now, we have another matrix $H'=(h'_{ij})\in\mathbb{M}_m(k)$.

\begin{lemma}\label{equivalent definition of homologcial determinant}
The homological determinant $\hdet\sigma=H'$.
\end{lemma}
\begin{proof} Let $i\in\{1,\cdots,m\}$. The natural projection $\mathrm{pr}_i:k^{\oplus m}\to k$ induces a natural projection $\widetilde{\mathrm{pr}}_i: P^\centerdot\otimes U\to P^\centerdot$ satisfying $\varepsilon_A\circ \widetilde{\mathrm{pr}}_i=\mathrm{pr}_i\circ \varepsilon^{\oplus m}_A$. Now we consider the following commutative diagram
$$
\xymatrix{
{^\varphi(P^\centerdot\otimes U)}\ar[rrr]^{\theta_i}\ar[rd]^{\varepsilon_A^{\oplus m}}\ar[dd]^{\theta}&  &  &  P^\centerdot\ar[ld]_{\varepsilon_A}\ar@{=}[dd]\\
  &  {_Ak}^{\oplus m}\ar[r]^{\mathrm{pr}_i}  &{_Ak} & \\
P^\centerdot\otimes U \ar[rrr]^{\widetilde{\mathrm{pr}_i}}\ar[ur]^{\varepsilon^{\oplus m}_A}&&&P^\centerdot.\ar[ul]_{\varepsilon_A}
}
$$
So $\varepsilon_A\circ \widetilde{\mathrm{pr}_i}\circ \theta=\varepsilon_A\circ\theta_i$. Because $\varepsilon_A:P^\centerdot\to {_Ak}$ is a quasi-isomorphism and $\uHom_A({^\varphi(P^\centerdot\otimes U)},-)$ preserves quasi-isomorphism, $\widetilde{\mathrm{pr}_i}\circ\theta-\theta_i$ is null homotopic. That is to say
$$
\uHom_A(\theta_i,k)=\uHom_A(\widetilde{\mathrm{pr}}_i\circ\theta,k):E(A)\to E(A)^{\oplus m}.
$$
However, it is obvious $\uHom_A(\widetilde{\mathrm{pr}}_i,k):E(A)\to E(A)^{\oplus m}$ is a natural injection of $i$-th position. Hence,
$$
\uHom_A(\theta_i,k)(\omega^*)=\uHom_A(\theta,k)\circ \uHom_A(\widetilde{\mathrm{pr}}_i,k)(\omega^*)=\uHom_A(\theta,k)(\omega_i^*).
$$
The proof is completed.
\end{proof}

Recall that the minimal free resolution (\ref{resolution of Bk}) of ${_Bk}$ is
\begin{equation}
Q^\centerdot = (\cdots \xlongrightarrow{d_Q^{-4}} B\otimes W_3\xlongrightarrow{d_Q^{-3}} B\otimes W_2 \xlongrightarrow{d_Q^{-2}} B\otimes W_1 \xlongrightarrow{d_Q^{-1}} B\to 0\to\cdots) \quad \xlongrightarrow{\varepsilon_B}\quad {_Bk},\tag{RB}
\end{equation}
where $W_1=k(-l_1)^{\oplus m}$ with the canonical coordinate basis $\{e_1,e_2,\cdots,e_m\}$. The dual basis $\{e^*_1,e^*_2,\cdots,e^*_m\}$ of $W_1^*$ can be seen as the basis elements of $E^1(B)$.

It is easy to know that graded $A$-module homomorphism $A\otimes e^*_i: A\otimes W_1\to A(-l_1)$ is also a graded right $A$-module homomorphism. Then we have a graded $A^e$-module homomorphism
$$
\Psi:=(A\otimes e^*_1,\cdots,A\otimes e^*_m):(A\otimes W_1)^\sigma\to (A(-l_1)^{\oplus m})^\sigma\cong (A\otimes U)^\sigma(-l_1).
$$
\begin{lemma}\label{isomorphism for twisted modules}
For any graded $A$-module $M$, there exists a natural isomorphism as graded $A$-modules
$$
(A\otimes U)^\sigma\otimes_A M\cong {^\varphi (M\otimes U)}.
$$
\end{lemma}
\begin{proof}
Let $\{u_1,u_2,\cdots,u_m\}$ be the canonical coordinate basis of $U$. Define
$$
\begin{array}{cclcll}
\zeta: &(A\otimes U)^\sigma \otimes_A M         &  \to   &A\otimes_A\,{^\varphi(M\otimes U)}&\xrightarrow{\cong} &{^\varphi(M\otimes U)}\\
         &(\sum_{i=1}^m a_i\otimes u_i)\otimes_Ax &\mapsto&\sum_{i=1}^m a_i\otimes_A (x\otimes u_i)&\mapsto  & \sum_{i=1}^m ( \sum_{j=1}^m(\varphi_{ji}(a_i)x\otimes u_j).
\end{array}
$$
It is easy to check it is a natural graded $A$-module isomorphism.
\end{proof}

Using the isomorphism in the lemma above, one obtains a morphism of cochain complexes in $\Ch(\Gr A)$:
$$
\xi_i:(A\otimes W_1)^\sigma\otimes_AP^\centerdot\xlongrightarrow{\Psi\otimes_AP^\centerdot} (A\otimes U)^\sigma(-l_1)\otimes_AP^\centerdot\xlongrightarrow{\cong}{^\varphi (P^\centerdot\otimes U)}(-l_1)\xlongrightarrow{\theta_i(-l_1)}
P^\centerdot(-l_1),
$$
where $i=1,\cdots,m$. By Lemma \ref{equivalent definition of homologcial determinant} and Lemma \ref{isomorphism for twisted modules}, we have an immediate result.
\begin{lemma}\label{representation of xi}
$\left((\xi_1)^{-h_A},\cdots,(\xi_m)^{-h_A}\right)^T=\hdet\sigma \cdot\left((A\otimes e^*_1)\otimes_AA(-l_A),\cdots,(A\otimes e^*_m)\otimes_AA(-l_A)\right)^T$.
\end{lemma}

Now it is the turn to show a result about the Yoneda product.
\begin{lemma}\label{E^hA(A) * E1(B)}
Let $f\in V^*_{h_A}$, then
$$
(1\otimes f)\cdot (e^*_1\otimes 1,\cdots, e^*_m\otimes 1)^T=\hdet \sigma\cdot (e^*_1\otimes f,\cdots, e^*_m\otimes f)^T.
$$
\end{lemma}
\begin{proof}
Write $\alpha:P^\centerdot\to {_Ak(-l_A)}[h_A]$ to be the morphism of cochain complexes corresponding to $f$, and $\beta_i:Q^\centerdot\to {_Bk}(-l_1)[1]$ to be the one corresponding to $e^*_i$ for $i=1,\cdots,m$. By Lemma \ref{Images of E(pi_A) and E(pi_B)} and Remark \ref{a note for E(A), E(B) and E(C)}, $E(\pi_A)(\alpha):F^\centerdot_\phi\otimes_AP^\centerdot\to {_Ck}(-l_A)[h_A]$ in $\Ch(\Gr C)$ corresponds to $1\otimes f\in W_0^*\otimes V^*_{h_A}$ induced by
$$
\pi_A\otimes_A(\varepsilon_A\otimes f):C\otimes_A(A\otimes V_{h_A}) \to {_Ck(-l_A)},
$$
and each $E(\pi_B)(\beta_i):F^\centerdot_\phi\otimes_AP^\centerdot\to {_Ck}(-l_1)[1]$ in $\Ch(\Gr C)$ corresponds to $e^*_i\otimes 1\in W^*_{1}\otimes V_0^*$ induced by
$$
(C\otimes W_1)^{\sigma}\otimes_AA\cong C\otimes W_1\xlongrightarrow{\pi_B\otimes W_{1}}B\otimes W_{1}\xlongrightarrow{\varepsilon_B\otimes e^*_i} {_Ck}(-l_1).
$$


The graded $(C,A)$-bimodule homomorphism $\pi_A\otimes W_1:(C\otimes W_1)^\sigma\to (A\otimes W_1)^\sigma$ induces a morphism of cochain complexes in $\Ch(\Gr\,(C\otimes A^o))$
$$
\eta: F^\centerdot_\phi\to (A\otimes W_1)^\sigma[1],
$$
since the differential $d_F=A\otimes^{\tau}d_Q$ and $Q^\centerdot$ is minimal.

Let $i$ be an arbitrary integer. Then we focus on the following diagram
$$
\xymatrix{
&    F^\centerdot_\phi\otimes_AP^\centerdot\ar[rr]^{\widetilde{\pi}_B}\ar[rrd]^{E(\pi_B)(\beta_i)}\ar[d]^{\gamma}\ar[ld]^{\eta\otimes_AP^\centerdot}      &&  Q^\centerdot\ar[d]^{\beta_i}        \\
\left((A\otimes W_1)^\sigma\otimes_AP^\centerdot\right)[1]\ar[rd]^{\xi_i[1]} & (F^\centerdot_\phi\otimes_AP^\centerdot)(-l_1)[1]\ar[rr]^{\varepsilon_C(-l_1)[1]}\ar[d]^{\widetilde{\pi}_A(-l_1)[1]}   &&   {k}(-l_1)[1]\\
& P^\centerdot(-l_1)[1]\ar[rru]_{\varepsilon_A(-l_1)[1]}\ar[d]^{\alpha(-l_1)[1]}&&\\
&  k(-l_1-l_A)[1+h_A],&
}
$$
where $\gamma$ is chosen to satisfy $\varepsilon_C(-l_1)[1]\circ \gamma=E(\pi_B)(\beta_i)=\beta_i\circ \widetilde{\pi}_B$. It is clear the right half of the diagram above is commutative, so $\varepsilon_A(-l_1)[1]\circ \widetilde{\pi_A}(-l_1)[1]\circ\gamma=\beta_i\circ\widetilde{\pi}_B.$
By a straightforward computation, one obtains $\varepsilon_A(-l_1)[1]\circ\xi_i[1]\circ (\eta\otimes_AP^\centerdot)=\beta_i\circ\widetilde{\pi}_B$, that is,
$$
\varepsilon_A(-l_1)[1]\circ \widetilde{\pi_A}(-l_1)[1]\circ\gamma=\varepsilon_A(-l_1)[1]\circ\xi_i[1]\circ (\eta\otimes_AP^\centerdot).
$$
Because $\varepsilon_A$ is a quasi-isomorphism, $\widetilde{\pi_A}(-l_1)[1]\circ \gamma$ is homotopic to $\xi_i[1]\circ (\eta\otimes_AP^\centerdot)$. Combined with the minimality of resolutions $F_\phi^\centerdot$ and $P^\centerdot,$ we have
$$
\alpha(-l_1)[1]\circ\widetilde{\pi_A}(-l_1)[1]\circ \gamma=\alpha(-l_1)[1]\circ\xi_i[1]\circ (\eta\otimes_AP^\centerdot).
$$

In the other hand, $E(\pi_A)(\alpha)\cdot E(\pi_B)(\beta_i)=\alpha(-l_1)[1]\circ\widetilde{\pi}_A(-l_1)[1]\circ \gamma=\alpha(-l_1)[1]\circ\xi_i[1]\circ (\eta\otimes_AP^\centerdot)$ and the nonzero component equals the composition of following morphisms
$$
(C\otimes W_1)^\sigma\otimes_A (A\otimes V_{h_A})\xlongrightarrow{(\pi_A\otimes W_1)\otimes_A(A\otimes V_{h_A})}(A\otimes W_1)^\sigma\otimes_A(A\otimes V_{h_A})\xlongrightarrow{(\xi_i)^{-h_A}}(A\otimes V_{h_A})(-l_1)\xlongrightarrow{(\varepsilon_A\otimes f)(-l_1)}{_Ck}(-l_1-l_A).
$$
Hence the result follows by Lemma \ref{representation of xi} and Remark \ref{a note for E(A), E(B) and E(C)}.
\end{proof}

\section{Nakayama Automorphisms}\label{Section Nakayama automorphisms}
With the preparations of results for Ext-algebras, we turn to describe Nakayama automorphisms of twisted tensor products. We still keep to the Hypothesis \ref{hypothesis on twisted tensor products} firstly. In addition, we assume connected graded algebras $A$ and $B$ are noetherian AS-regular algebras of type $(h_A,l_A)$ and $(h_B,l_B)$ respectively which are both generated in degree $1$, and the graded algebra homomorphism $\sigma$ is invertible.

By Theorem \ref{AS-regularity of twisted tensor products}, $A\otimes^{\tau}\!B$ is AS-regular and has Nakayama automorphism. We study the Nakayama automorphism of $A\otimes^{\tau}\!B$ by taking advantage of the classical Nakayama automorphism of Ext-algebra $E(A\otimes^{\tau}\!B)$. It depends on the Yoneda product of  $E(A\otimes^{\tau}\!B)\cong E(B)\otimes^{\tau_E}\!E(A)$ for the bigraded twisting map ${\tau_E}$ occurred in Theorem \ref{E(C) is a twisted tensor product of E(B) and E(A)}. The last two cases of Section \ref{Section ext-algebras} make us be able to describe the bigraded linear map ${\tau_E}$ restricted on $E^1(A)\otimes E^{h_B}(B)$ and $E^{h_A}(A)\otimes E^1(B)$. Fortunately, it suffices to  realize the goal for Nakayama automorphisms.

Under the assumption, $E^1(A)=A_1^*$ and $E^1(B)=B_1^*$. The $k$-dual of determinant $(\det\sigma)^*$ is an automorphism of $E^1(A)$, and  $E^1(B)$ has a basis $\{y_1^*,\cdots,y_m^*\}$.

\begin{lemma}\label{form of RE}
Let $f_1\in E^1(A),f_{h_A}\in E^{h_A}(A)$ and $g\in E^{h_B}(B)$, then
\begin{align*}
&{{\tau_E}}(f_1\otimes g)=(-1)^{h_B}g\otimes \left(\det\sigma_{\,|A_1}\right)^*(f_1),\\
&{{\tau_E}}\left(
\begin{array}{c}
f_{h_A}\otimes y_1^*\\
f_{h_A}\otimes y_2^*\\
\vdots\\
f_{h_A}\otimes y_m^*
\end{array}
\right)
=
(-1)^{h_A}\hdet\sigma\cdot\left(
\begin{array}{c}
y_1^*\otimes f_{h_A}\\
y_2^*\otimes f_{h_A}\\
\vdots\\
y_m^*\otimes f_{h_A}
\end{array}
\right).
\end{align*}
\end{lemma}
\begin{proof}
Write $m_{E}$ to be the Yoneda product of the Ext-algebra $E(A\otimes^{\tau}\!B)$. By Lemma \ref{E^1(A) *E^hB(B)}, we have
$$
m_{E}\circ (E(\pi_A)\otimes E(\pi_B))(f_1\otimes g)=(1\otimes f_1)\cdot (g\otimes 1)=g\otimes \left(\det\sigma_{\,|A_1}\right)^*(f_1)\in E(A\otimes^{\tau}\!B).
$$
In the other hand, following from Lemma \ref{The Yoneda product of E(pi_B) and E(pi_A)}, one obtains
\begin{eqnarray*}
m_{E}\circ (E(\pi_B)\otimes E(\pi_A))(g\otimes \left(\det\sigma_{\,|A_1}\right)^*(f_1))&=&(g\otimes1)\cdot(1\otimes \left(\det\sigma_{\,|A_1}\right)^*(f_1))\\
&=&(-1)^{h_B}g\otimes \left(\det\sigma_{\,|A_1}\right)^*(f_1)\in E(A\otimes^{\tau}\!B).
\end{eqnarray*}
Hence, we have
$$
\tau_E(f_1\otimes g)=\Big{(}m_{E}\circ (E(\pi_B)\otimes E(\pi_A))\Big{)}^{-1}\circ m_{E}\circ (E(\pi_A)\otimes E(\pi_B))(f_1\otimes g)=(-1)^{h_B}g\otimes \left(\det\sigma_{\,|A_1}\right)^*(f_1),
$$
by Theorem \ref{E(C) is a twisted tensor product of E(B) and E(A)}. The second equation follows from Lemma \ref{E^hA(A) * E1(B)} similarly.
\end{proof}

\begin{theorem}\label{Nakayama automorphism of twisted tensor products}
Let $A=k\langle X\rangle/(G_A)$ and $B=k\langle Y\rangle/(G_B)$ be both noetherian AS-regular algebras generated in degree $1$, where $X=\{x_i\}_{i=1}^n$ and $Y=\{y_j\}_{j=1}^m$ are minimal generating sets of $A$ and $B$ respectively. Let $\tau=(\sigma,0):B\otimes A\to A\otimes B$ be a graded twisting map such that $A\otimes^\tau\! B$ is noetherian. Suppose $B$ has a pure resolution and $\sigma$ is invertible, then the Nakayama automorphism $\mu_{A\otimes^{\tau}B}$ satisfies
\begin{align*}
&{\mu_{A\otimes^{\tau}B}}_{|A}=({\det} \sigma)^{-1}\circ\mu_A,\\
&{\mu_{A\otimes^{\tau}B}}_{|B}
\left(
\begin{array}{c}
y_1\\
y_2\\
\vdots\\
y_m
\end{array}
\right)
=
\hdet\sigma\cdot\mu_B\left(
\begin{array}{c}
y_1\\
y_2\\
\vdots\\
y_m
\end{array}
\right).
\end{align*}
\end{theorem}
\begin{proof} Write $C=A\otimes^{\tau}\!B$. Denote the Ext-algebras of $A,B$ and $C$ by $E(A),E(B)$ and $E(C)$ respectively. Let $\{x_i^*\}_{i=1}^n$ and $\{y_j^*\}_{j=1}^m$ be dual base of $E^1(A)$ and $E^1(B)$ respectively. Write $\nu=\det \sigma$.

Assume $A$ and $B$ are AS-regular algebras of type $(h_A,l_A)$ and $(h_B,l_B)$ respectively.  By Theorem \ref{AS-regularity of twisted tensor products},  $C$ is AS-regular of type $(h_A+h_B,l_A+l_B)$. Following from  Theorem \ref{the properties between ext algebras and AS regular algebras}(a), $E(A)$, $E(B)$ and $E(C)$ are Frobenius algebras of type $(-h_A,l_A)$, $(-h_B,l_B)$  and $(-h_A-h_B,l_A+l_B)$ respectively. Denote the (classical) Nakayama automorphisms of $E(A)$, $E(B)$ and $E(C)$ by $\mu_{E(A)}$, $\mu_{E(B)}$ and $\mu_{E(C)}$ respectively.

By Theorem \ref{E(C) is a twisted tensor product of E(B) and E(A)}, there exists a bigraded linear map  ${\tau_E}:E(A)\otimes E(B)\to E(B)\otimes E(A)$ such that $E(C)=E(B)\otimes^{\tau_E}\!E(A)$. Since $A$ and $B$ are generated in degree $1$, $\{1\otimes^{\tau_E} x_1^*,\cdots,1\otimes^{\tau_E} x_n^*,y_1^*\otimes^{\tau_E} 1,\cdots,y_m^*\otimes^{\tau_E} 1\}$ is a basis of $E^1(C)$.

Firstly, we focus on $\mu_{E(C)}(1\otimes^{\tau_E} x_i^*)$. Let $\omega_B$ be a basis of $E^{h_B}(B)$, and $\{f_1,\cdots,f_n\}$ be a basis of $E^{h_A-1}(A)$. Since the degree of bigraded linear form for $E(C)$ is $(-h_A-h_B,l_A+l_B)$, the only possible nonzero values come from
$$
\langle1\otimes^{\tau_E}(\nu^{-1})^*(x_i^*), \omega_B\otimes^{\tau_E} f_j\rangle =\langle\left(1\otimes^{\tau_E}(\nu^{-1})^*(x_i^*)\right)(\omega_B\otimes^{\tau_E}1), 1\otimes^{\tau_E} f_j\rangle ,
$$
following from the associativity of bilinear form for any $i,j=1,\cdots,n$. By Lemma \ref{form of RE}, we have
$$
\left(1\otimes^{\tau_E}(\nu^{-1})^*(x_i^*)\right)(\omega_B\otimes^{\tau_E}1)=\omega_B\otimes^{\tau_E} (-1)^{h_B}x_i^*.
$$
Using the associativity again,
\begin{eqnarray*}
\langle1\otimes^{\tau_E}(\nu^{-1})^*(x_i^*), \omega_B\otimes^{\tau_E} f_j\rangle &=&\langle\omega_B\otimes^{\tau_E} (-1)^{h_B}x_i^*, 1\otimes^{\tau_E} f_j\rangle \\
&=&(-1)^{h_B}\langle\omega_B\otimes^{\tau_E}1, 1\otimes^{\tau_E} x_i^*f_j\rangle \\
&=&(-1)^{h_B+h_A-1}\langle\omega_B\otimes^{\tau_E}1, 1\otimes^{\tau_E} f_j\mu_{E(A)}^{-1}(x_i^*)\rangle \\
&=&(-1)^{h_B+h_A-1}\langle\omega_B\otimes^{\tau_E}f_j, 1\otimes^{\tau_E} \mu_{E(A)}^{-1}(x_i^*)\rangle .
\end{eqnarray*}
So $\mu_{E(C)}(1\otimes^{\tau_E}x_i^*)=1\otimes^{\tau_E}(\nu^{-1})^*\mu_{E(A)}(x_i^*)$. It implies
$$
{\mu_{E(C)}}_{|E^1(A)}=(\nu^{-1})^*{\mu_{E(A)}}_{|E^1(A)}.
$$
By Theorem \ref{the properties between ext algebras and AS regular algebras}(b) and the noetherianess of $C$,
\begin{equation*}
{\mu_C}_{|A_1}=({\mu_{E(C)}}_{|E^1(A)})^*=\left((\nu^{-1})^*\circ\mu_{E(A)}\right)^*=\mu_{E(A)}^*\circ\nu^{-1}=\mu_A\circ{\nu^{-1}}_{|A_1}.
\end{equation*}
In conclusion, ${\mu_C}_{|A}=\mu_A\circ (\det \sigma)^{-1}=(\det \sigma)^{-1} \circ\mu_A$, since the Nakayama automorphism is in the center of automorphism group of algebra by \cite[Theorem 0.6]{LMZ}.

Now we turn to consider $\mu_{E(C)}(y_j^*\otimes^{\tau_E} 1)$. Let $\omega_A$ be a basis of $E^{h_A}(A)$, and $\{g_1,\cdots,g_m\}$ be a basis of $E^{h_B-1}(B)$.  Using the associativity of bilinear form $\langle\cdot,\cdot\rangle $ of $E(C)$, we have the only possible nonzero values
$$
\langle g_i\otimes^{\tau_E}\omega_A, y_j^*\otimes^{\tau_E} 1\rangle =\langle g_i\otimes^{\tau_E}1,(1\otimes^{\tau_E}\omega_A)(y_j^*\otimes^{\tau_E} 1)\rangle ,
$$
for any $i,j=1,2,\cdots,m$. By Lemma \ref{form of RE},
$$
(1\otimes^{\tau_E}\omega_A)(y_1^*\otimes^{\tau_E} 1,\cdots,y_m^*\otimes^{\tau_E} 1)^T=(-1)^{h_A}\hdet\sigma\cdot (y_1^*\otimes^{\tau_E}\omega_A,\cdots,y_m^*\otimes^{\tau_E}\omega_A)^T.
$$
By the associativity again,
\begin{align*}
&\;\left(\langle g_i\otimes^{\tau_E}\omega_A, y_1^*\otimes^{\tau_E} 1\rangle,\cdots,\langle g_i\otimes^{\tau_E}\omega_A, y_m^*\otimes^{\tau_E} 1\rangle \right)^T \\
=&\; (-1)^{h_A}\hdet\sigma\cdot\left(\langle g_i\otimes^{\tau_E}1,y_1^*\otimes^{\tau_E}\omega_A\rangle,\cdots, \langle g_i\otimes^{\tau_E}1,y_m^*\otimes^{\tau_E}\omega_A\rangle\right)^T \\
=&\; (-1)^{h_A}\hdet\sigma\cdot\left(\langle g_iy^*_1\otimes^{\tau_E}1,1\otimes^{\tau_E}\omega_A\rangle,\cdots, \langle g_iy^*_m\otimes^{\tau_E}1,1\otimes^{\tau_E}\omega_A\rangle\right)^T  \\
=&\;(-1)^{h_A+h_B-1}\hdet\sigma\cdot\left(\langle \mu_{E(B)}(y_1^*)g_i\otimes^{\tau_E}1,1\otimes^{\tau_E}\omega_A\rangle,\cdots, \langle \mu_{E(B)}(y_m^*)g_i\otimes^{\tau_E}1,1\otimes^{\tau_E}\omega_A\rangle\right)^T\\
=&\;(-1)^{h_A+h_B-1}\hdet\sigma\cdot\left(\langle \mu_{E(B)}(y_1^*)\otimes^{\tau_E}1,g_i\otimes^{\tau_E}\omega_A\rangle,\cdots, \langle \mu_{E(B)}(y_m^*)\otimes^{\tau_E}1,g_i\otimes^{\tau_E}\omega_A\rangle\right)^T.
\end{align*}
Thus $\left(\mu_{E(C)}(y_1^*\otimes^{\tau_E} 1),\cdots,\mu_{E(C)}(y_m^*\otimes^{\tau_E} 1)\right)^T=\hdet\sigma\cdot\left(\mu_{E(B)}(y_1^*)\otimes^{\tau_E} 1,\cdots,\mu_{E(B)}(y_m^*)\otimes^{\tau_E} 1\right)^T$, and
$$
{\mu_{E(C)}}_{|E^1(B)}(y_1^*,\cdots,y_m^*)^T=\hdet\sigma\cdot\mu_{E(B)}(y_1^*,\cdots,y_m^*)^T.
$$
By Theorem \ref{the properties between ext algebras and AS regular algebras}(b), the noetherianess of $C$ and Proposition \ref{hdet is invertible},
\begin{align*}
&{\mu_C}_{|B}
\left(
\begin{array}{c}
y_1\\
y_2\\
\vdots\\
y_m
\end{array}
\right)
=
\hdet\sigma\cdot\mu_B\left(
\begin{array}{c}
y_1\\
y_2\\
\vdots\\
y_m
\end{array}
\right).
\end{align*}
The proof is completed.
\end{proof}

Finally, we assume the graded twisting map $\tau=(\sigma,\delta):B\otimes A\to A\otimes B$ with nonzero $\sigma$-derivation $\delta$. In this case, the related graded linear map $\bar{\tau}=(\sigma,0):B\otimes A\to A\otimes B$ with zero $\sigma$-derivation is also a graded twisting map. We link the two twisted tensor products by a filtration of $A\otimes^{\tau}B$.

For any monomial $w$ in $k\langle X\rangle$ and $k\langle Y\rangle$, we have a canonical definition of length for $w$, denoted by $l(w)$. Give a monid homomorphism $l'$ from $k\langle X,Y\rangle$ to $\mathbb{N}$ satisfying
$$l'(x)=l(x)\quad\text{ and }\quad l'(y)=l(y)+1,$$
where $x\in X$ and $y\in Y$. So there is a filtration on $k\langle X,Y\rangle$ such that $F_i\,\Big{(}k\langle X,Y\rangle\Big{)}$ is a vector space spanned by monomials $w$ in $k\langle X,Y\rangle$ with $l'(w)\leq i$ for $i\in \mathbb{Z}$. Since $A\otimes^{\tau}B\cong k\langle X,Y\rangle/I$ for some ideal $I$ of $k\langle X,Y\rangle$, there is a natural filtration
$$
F_i(A\otimes^{\tau}B):=\frac{F_i\,\Big{(}k\langle X,Y\rangle\Big{)}+I}{I},\quad \forall i\in\mathbb{Z}.
$$
The associated graded algebra $\gr (A\otimes^{\tau}\!B)\cong A\otimes^{\bar{\tau}}\!B$ as graded algebras. We say $A\otimes^{\bar{\tau}}\!B$ is the associated twisted tensor product of $A\otimes^{\tau}\!B$.

\begin{theorem}\label{Nakayama automorphism of twisted tensor products with nonzero delta}
Let $A=k\langle X\rangle/(G_A)$ and $B=k\langle Y\rangle/(G_B)$ be both noetherian AS-regular algebras generated in degree $1$, where $Y=\{y_i\}_{i=1}^m$ is the minimal generating set of $B$. Let $\tau=(\sigma,\delta):B\otimes A\to A\otimes B$ be a graded twisting map such that the associated twisted tensor product $A\otimes^{\bar{\tau}}\!B$ is noetherian. Suppose $B$ has a pure resolution and $\sigma$ is invertible, then Nakayama automorphism $\mu_{A\otimes^{\tau}B}$ satisfies
\begin{align*}
&{\mu_{A\otimes^{\tau}B}}_{|A}=({\det} \sigma)^{-1}\circ \mu_A,\\
&{\mu_{A\otimes^{\tau}B}}
\left(
\begin{array}{c}
1\otimes^{\tau}y_1\\
1\otimes^{\tau}y_2\\
\vdots\\
1\otimes^{\tau}y_m
\end{array}
\right)
=
\hdet\sigma\cdot(A\otimes^{\tau}\!\mu_B)\left(
\begin{array}{c}
1\otimes^{\tau}y_1\\
1\otimes^{\tau}y_2\\
\vdots\\
1\otimes^{\tau}y_m
\end{array}
\right)
+
\left(
\begin{array}{c}
a_1\otimes^{\tau}1\\
a_2\otimes^{\tau}1\\
\vdots\\
a_m\otimes^{\tau}1
\end{array}
\right),
\end{align*}
for some $a_1,a_2,\cdots,a_m\in A$.
\end{theorem}
\begin{proof}
By the discussion above this theorem, there is a filtration of $A\otimes^{\tau}\!B$ such that the associated graded algebra $\gr (A\otimes^{\tau}\!B)\cong A\otimes^{\bar{\tau}}\!B$ where $\bar{\tau}=(\sigma,0)$. The Nakayama automorphism $\mu_{A\otimes^{\tau}B}$ of $A\otimes^{\tau}\!B$ is filtered and $\gr\, \mu_{A\otimes^{\tau}B}=\mu_{A\otimes^{\bar{\tau}}B}$ by \cite[Lemma 5]{V}. The result follows from Theorem \ref{Nakayama automorphism of twisted tensor products} and the construction of filtration.
\end{proof}

\begin{remark}\label{Nakayama automorphisms for Koszul assumption}
From Proposition \ref{minimal resolution of Ck}, it is easy to know the twisted tensor product of Koszul algebras is also Koszul. If the noetherian condition is replaced by Koszul, the Theorem \ref{the properties between ext algebras and AS regular algebras}(b) still hold by \cite[Theorem 1.3]{BM}. So we can release the noetherian conditions in Theorem \ref{Nakayama automorphism of twisted tensor products} and \ref{Nakayama automorphism of twisted tensor products with nonzero delta} if the connected graded algebras $A$ and $B$ are Koszul algebras.
\end{remark}

\section{Examples}\label{Section Examples}
We conclude this paper with some applications of Theorem \ref{Nakayama automorphism of twisted tensor products with nonzero delta}.
\begin{example}\label{Example:Nakayama automorphisms of Ore extensions}
The classical Ore extension $A[z;\sigma,\delta]$ is a twisted tensor product $A\otimes^{\tau}k[z]$ with
$$
\tau(z\otimes a)=\sigma(a)\otimes z+\delta(a)\otimes 1,\quad \forall a\in A,
$$
that is, $\tau=(\sigma,\delta)$. In this case, $\det\sigma=\sigma$ and $\hdet\sigma$ is just the one defined in \cite{JZ}. The Nakayama of polynomial algebra $k[z]$ is the identity map. If $A$ is a noetherian AS-regular algebra with Nakayama automorphism $\mu_A$ generated in degree $1$ and $\sigma$ is an automorphism of $A$, then the Nakayama automorphism $\mu$ of $A[z;\sigma,\delta]$ satisfies
\begin{eqnarray*}
&&\mu(a)=\sigma^{-1}\circ\mu_A(a),\quad\forall a\in A\\
&&\mu(z)=(\hdet\sigma) z+a',
\end{eqnarray*}
for some $a'\in A$, by Theorem \ref{Nakayama automorphism of twisted tensor products with nonzero delta}. It coincides with the result \cite[Theorem 2]{LWW}.

\end{example}

\begin{example}\label{Example:Nakayama automorphisms of double Ore extensions}
The second example is double Ore extension with zero tails $A_P[y_1,y_2;\sigma,\delta]$, where $P=\{p_{12},p_{11}\}\subseteq k$, algebra homomorphism $\sigma:A\to \mathbb{M}_2(A)$ and $\sigma$-derivation $\delta:A\to A^{\oplus 2}$ (see \cite{ZZ1} for details). It is a twisted tensor product of $A$ and $B=k\langle y_1,y_2\rangle/(y_2y_1-p_{12}y_1y_2-p_{11}y_1^2)$ with
$$
\tau(a\otimes y_i)=\sigma_{i1}(a)\otimes y_1+\sigma_{i2}(a)\otimes y_2+\delta_i(a)\otimes 1,\quad\forall a\in A,i=1,2.
$$
So $\tau=(\sigma,\delta)$, and $\det\sigma=-p_{11}\sigma_{12}\sigma_{11}+\sigma_{22}\sigma_{11}-p_{12}\sigma_{12}\sigma_{21}$.
If $p_{12}$ is nonzero, $B$ is Koszul AS-regular with Nakayama automorphism $\mu_B$:
$$
\mu_{B}
\left(
\begin{array}{c}
y_1\\
y_2
\end{array}
\right)
=
\left(
\begin{array}{cc}
p_{12}^{-1}&0\\
p_{11}(1+p_{12}^{-1})&p_{12}
\end{array}
\right)
\left(
\begin{array}{c}
y_1\\
y_2
\end{array}
\right),
$$
and $\det\sigma$ also equals $-p_{12}^{-1}p_{11}\sigma_{11}\sigma_{12}-p_{12}^{-1}\sigma_{21}\sigma_{12}+\sigma_{11}\sigma_{12}$.  If $A$ is a noetherian AS-regular algebra with Nakayama automorphism $\mu_A$ generated in degree $1$, the associated trimmed double Ore extension $A_P[y_1,y_2;\sigma]$ is noetherian and $\sigma$ is invertible, by Theorem \ref{Nakayama automorphism of twisted tensor products with nonzero delta}, we have Nakayama automorphism $\mu$ of $A_P[y_1,y_2;\sigma,\delta]$ satisfying
\begin{eqnarray*}
&&\mu(a)=(\det \sigma)^{-1}\circ\mu_A(a),\quad\forall a\in A\\
&&\mu
\left(
\begin{array}{c}
	y_1\\
	y_2
\end{array}
\right)=
\hdet\sigma\cdot
\left(
\begin{array}{cc}
	p_{12}^{-1}&0\\
	p_{11}(1+p_{12}^{-1})&p_{12}
\end{array}
\right)
\left(
\begin{array}{c}
	y_1\\
	y_2
\end{array}
\right)
+
\left(
\begin{array}{c}
	a_1\\
	a_2
\end{array}
\right),
\end{eqnarray*}
for some $a_1,a_2\in A$. In particular, if $A$ is Koszul and $\delta=0$, this Nakayama automorphism of trimmed double Ore extension $A_P[y_1,y_2;\sigma]$ is the same to the one obtained in \cite[Theorem 2]{ZVZ} by Remark \ref{Nakayama automorphisms for Koszul assumption}.
\end{example}

\begin{example}\label{Example:twisted tensor product of special down-up algebras}
Let $k$ be a field of characteristic 0. We consider the Nakayama automorphism of the example constructed in \cite[Example 4.3]{WSL}. Let $A=k\langle x_{1},x_{2}\rangle/(f_{1},f_{2})$, where $f_{1}=x_{1}^{2}x_{2}-x_{2}x_{1}^{2}$ and $f_{2}=x_{1}x_{2}^{2}-x_{2}^{2}x_{1}$. In fact, $A$ is a  noetherian $3$-Koszul AS-regular algebra, and the Nakayama automorphism $\mu_A$ of $A$ is
\begin{eqnarray*}
\mu_A(x_1)=-x_1,\quad \mu_A(x_2)=-x_2.
\end{eqnarray*}
Define a graded normal linear map $\tau=(\sigma,0)$ from $A\otimes A$ to $A\otimes A$
by
$$
\begin{array}{ll}
\tau(x_{1}\otimes x_{1})=px_{2}\otimes x_{1}+px_{2}\otimes x_{2},\\
\tau(x_{1}\otimes x_{2})=px_{1}\otimes x_{1}+px_{1}\otimes x_{2},\\
\tau(x_{2}\otimes x_{1})=px_{2}\otimes x_{1}-px_{2}\otimes x_{2},\\
\tau(x_{2}\otimes x_{2})=px_{1}\otimes x_{1}-px_{1}\otimes x_{2},
\end{array}
$$
where $p\in k^{\times}$. The algebra homomorphism $\sigma:A\to\mathbb{M}_2(A)$ is invertible with inverse $\varphi$ satisfying:
\begin{eqnarray*}
&\sigma(x_1)=
p\left(
\begin{array}{cc}
x_2&x_2\\
x_2&-x_2
\end{array}
\right),
\quad
&\sigma(x_2)=
p\left(
\begin{array}{cc}
x_1&x_1\\
x_1&-x_1
\end{array}
\right).
\end{eqnarray*}
The twisted tensor product $A\otimes^{\tau}\!A$ is isomorphic to the algebra $C:=k\langle x_1,x_2,y_1,y_2\rangle/(G)$, where $G=\{g_i\}_{i=1}^8$, and
$$\begin{array}{llll}
g_{1}=x_{1}^{2}x_{2}-x_{2}x_{1}^{2}, &&& g_{2}=x_{1}x_{2}^{2}-x_{2}^{2}x_{1},\\
g_{3}=y_{1}^{2}y_{2}-y_{2}y_{1}^{2}, &&& g_{4}=y_{1}y_{2}^{2}-y_{2}^{2}y_{1},\\
g_{5}=y_1x_1-px_2y_1-px_2y_2, &&& g_{6}=y_1x_2-px_1y_1-px_1y_2,\\
g_{7}=y_2x_1-px_2y_1+px_2y_2, &&& g_{8}=y_2x_2-px_1y_1+px_1y_2.
\end{array}$$
The algebra $C$ is also a noetherian AS-regular algebra by Theorem \ref{AS-regularity of twisted tensor products}. We compute the Nakayama automorphism of $C$ in the following.

A minimal free resolution of $_Ak$ is
$$
P^\centerdot: 0\to A(-4)\xlongrightarrow{d^{-3}} A(-3)^{\oplus 2}\xlongrightarrow{d^{-2}} A(-1)^{\oplus 2}\xlongrightarrow{d^{-1}} A\to {_Ak}\to0,
$$
where
\begin{eqnarray*}
d^{-1}=\left(\begin{array}{c}x_1\\x_2\end{array}\right),\quad
d^{-2}=\left(\begin{array}{cc}-x_2x_1&x_1^2\\
-x_2^2&x_1x_2\end{array}\right),\quad
d^{-3}=\left(\begin{array}{cc}-x_2&x_1\end{array}\right).
\end{eqnarray*}

Firstly, we calculate the determinant of $\sigma$. We construct two algebra homomorphisms $\phi_2:A\to\mathbb{M}_2(A)$ and $\det\sigma:A\to A$ to make complex $A\otimes^{\tau}P^\centerdot$ exact in $\GrMod\,(C\otimes A^o)$ as in Theorem \ref{resolution of A with bimodule structure}, that is,
$$
0\to C(-4)^{\det\sigma}\to \left(C(-3)^{\oplus2}\right)^{\phi_2}\to \left(C(-1)^{\oplus2}\right)^\sigma\to C\to A\to0,
$$
where
\begin{eqnarray*}
&&\phi_2(x_1)=-2p^3\left(\begin{array}{cc}x_2&x_2\\x_2&-x_2\end{array}\right),\quad
\phi_2(x_2)=-2p^3\left(\begin{array}{cc}x_1&x_1\\x_1&-x_1\end{array}\right),\\
&&\det\sigma (x_1)=4p^4x_1,\quad\quad\quad\quad\quad \det\sigma (x_2)=4p^4x_2.
\end{eqnarray*}

Now we turn to compute the homological determinant of $\sigma$. We have the following commutative diagram
$$
\xymatrix{
0\ar[r]&
\leftidx{^\varphi}{\left(A(-4)^{\oplus2}\right)}\ar[r]^(0.47){(d^{-3})^{\oplus2}}\ar[d]^{\Phi^{-3}}&
\leftidx{^\varphi}{\left(\left(A(-3)^{\oplus2}\right)^{\oplus2}\right)}\ar[r]^(0.5){(d^{-2})^{\oplus2}}\ar[d]^{\Phi^{-2}}&
\leftidx{^\varphi}{\left(\left(A(-1)^{\oplus2}\right)^{\oplus2}\right)}\ar[r]^(0.6){(d^{-1})^{\oplus2}}\ar[d]^{\Phi^{-1}}&
\leftidx{^\varphi}{\left(A^{\oplus2}\right)}\ar[r]\ar[d]^{\Phi^0}&
k^{\oplus2}\ar[r]\ar@{=}[d]&
0\\
0\ar[r]&
A(-4)^{\oplus2}\ar[r]^(0.47){(d^{-3})^{\oplus2}}&
\left(A(-3)^{\oplus2}\right)^{\oplus2}\ar[r]^(0.5){(d^{-2})^{\oplus2}}&
\left(A(-1)^{\oplus2}\right)^{\oplus2}\ar[r]^(0.6){(d^{-1})^{\oplus2}}&
A^{\oplus2}\ar[r]&
k^{\oplus2}\ar[r]&
0,\\
}
$$
where
\begin{align*}
&\Phi^0
\left(\hskip-1mm
\begin{array}{c}
a_1\\a_2
\end{array}
\hskip-1mm\right)
=
\left(\hskip-1mm
\begin{array}{c}
\sigma_{11}(a_1)+\sigma_{21}(a_2)\\
\sigma_{12}(a_1)+\sigma_{22}(a_2)
\end{array}\hskip-1mm\right),\\
&\Phi^{-1}
\left(\hskip-1mm
\begin{array}{c}
(a_1,a_2)\\
(b_1,b_2)
\end{array}
\hskip-1mm\right)
=
p\left(\hskip-1mm
\begin{array}{c}
(\sigma_{11}(a_2)+\sigma_{12}(a_2)+\sigma_{21}(b_2)+\sigma_{22}(b_2),
\sigma_{11}(a_1)+\sigma_{12}(a_1)+\sigma_{21}(b_1)+\sigma_{22}(b_1))\\
(\sigma_{11}(a_2)-\sigma_{12}(a_2)+\sigma_{21}(b_2)-\sigma_{22}(b_2),
\sigma_{11}(a_1)-\sigma_{12}(a_1)+\sigma_{21}(b_1)-\sigma_{22}(b_1))
\end{array}
\hskip-1mm\right),\\
&\Phi^{-2}
\left(\hskip-1mm
\begin{array}{c}
(a_1,a_2)\\
(b_1,b_2)
\end{array}
\hskip-1mm\right)
=
-2p^3\left(\hskip-1mm
\begin{array}{c}
(\sigma_{11}(a_2)+\sigma_{12}(a_2)+\sigma_{21}(b_2)+\sigma_{22}(b_2),
\sigma_{11}(a_1)+\sigma_{12}(a_1)+\sigma_{21}(b_1)+\sigma_{22}(b_1))\\
(\sigma_{11}(a_2)-\sigma_{12}(a_2)+\sigma_{21}(b_2)-\sigma_{22}(b_2),
\sigma_{11}(a_1)-\sigma_{12}(a_1)+\sigma_{21}(b_1)-\sigma_{22}(b_1))
\end{array}
\hskip-1mm\right),\\
&\Phi^{-3}
\left(\hskip-1mm
\begin{array}{c}
a_1\\a_2
\end{array}
\hskip-1mm\right)
=4p^4
\left(\hskip-1mm
\begin{array}{c}
\sigma_{11}(a_1)+\sigma_{21}(a_2)\\
\sigma_{12}(a_1)+\sigma_{22}(a_2)
\end{array}\hskip-1mm\right),
\end{align*}
for $a_1,a_2,b_1,b_2\in A$. By the definition, we have
$$
\hdet \sigma=\left(\begin{array}{cc}4p^4&0\\0&4p^4\end{array}\right).
$$
By Theorem \ref{Nakayama automorphism of twisted tensor products with nonzero delta}, we have the Nakayama automorphism $\mu_C$ satisfying
$$
\begin{array}{ll}
\mu_C(x_1)=-(4p^4)^{-1}x_1,\quad&\mu_C(x_2)=-(4p^4)^{-1}x_2,\\
\mu_C(y_1)=-4p^4y_1,\quad&\mu_C(y_2)=-4p^{4}y_2.
\end{array}
$$
\end{example}

\vskip7mm

\noindent {\bf Acknowledgments.} Y. Shen is supported by NSFC (Grant No.11626215) and Science Foundation of Zhejiang Sci-Tech University (ZSTU) under Grant No.16062066-Y. G.-S. Zhou is supported by NSFC (Grant No. 11601480). D.-M. Lu is supported by NSFC (Grant No. 11671351).
\vskip5mm


\vskip7mm


\begin{thebibliography}{10}

\bibitem{AS}
M. Artin and W. Schelter, \emph{Graded algebras of global dimension
$3$}, Adv. Math., \textbf{66(2)} (1987), 171--216.

\bibitem{BM}
R. Berger and N. Marconnet, \emph{Koszul and Gorenstein properties for homogeneous
algebras}, Algebr. Represent. Theory, \textbf{9(1)} (2006), 67--97.


\bibitem{CIMZ}
S. Caenepeel, B. Ion, G. Militaru and S. Zhu,  {\it The factorization problem and the smash biproduct of algebras and coalgebras}, Algebr. Represent. Theory, \textbf{3(1)} (2000), 19--42.

\bibitem{CSV}
A. \v{C}ap, H. Schichl and J. Van\v{z}ura, \emph{On twisted tensor products of algebras}, Comm. Algebra, \textbf{23(12)} (1995), 4701--4735.

\bibitem{CWZ}
K. Chan, C Walton and J. J. Zhang, \emph{Hopf actions and Nakayama automorphisms}, J. Alg., \textbf{409} (2014), 26--53.



\bibitem{JZ}
P. J\o gensen and J. J. Zhang, \emph{Gourmet's guide to gorensteiness}, Adv. Math., \textbf{151} (2000), 313--345.


\bibitem{LWW}
L.-Y. Liu, S.-Q. Wang and Q.-S. Wu, \emph{Twisted Calabi-Yau property of Ore extensions}, J. Noncommut. Geom., \textbf{7} (2014), 587--609.



\bibitem{LPWZ4}
D.-M. Lu, J. H. Palmieri, Q.-S. Wu and J. J. Zhang, \emph{Kosuzl equivalences in $A_\infty$-algebras}, New York J. Math., \textbf{14} (2008), 325--378.


\bibitem{LMZ}
J.-F. L\"u, X.-F. Mao and J. J. Zhang, \emph{Nakayama automorphism and applications}, Trans. Amer. Math. Soc., \textbf{369} (2017), 2425--2460.



\bibitem{RRZ1}M. Reyes, D. Rogalski, and J. J. Zhang, \emph{Skew Calabi-Yau algebras and homological identities}, Adv. Math.,\textbf{264} (2014), 308--354.

\bibitem{RRZ2}M. Reyes, D. Rogalski, and J. J. Zhang, \emph{Skew Calabi-Yau triangulated categories and Frobenius Ext-algebras.}, to appear in Trans. Amer. Math. Soc..


\bibitem{SWZ}
Y. Shen, X. Wang and G.-S. Zhou, \emph{Ext-algebras of graded skew extensions and $A_\infty$-structures}, arXiv:1707.01610, 2017.




\bibitem{Sm}
S. P. Smith, \emph{Some finite dimensional algebras related to elliptic curves}, CMS Conf. Proc., \textbf{19} (1996), 315--348.

\bibitem{SZ}
D. R. Stephenson and J. J. Zhang, \emph{Growth of graded Noetherian rings}, Proc. Amer. Math. Soc., \textbf{125(6)} (1997), 1593--1605.


\bibitem{VV}
A. Van Daele and S. Van Keer, \emph{The Yang-Baxter and pentagon equation}, Compositio Math., \textbf{91(2)} (1994), 201--221.

\bibitem{V}
M. Van den Bergh, \emph{A relation between Hochschild homology and cohomology for Gorenstein rings}, Proc. Amer. Math. Soc., \textbf{126} (1998), 1345--1348.

\bibitem{WW}
C. Walton and S. Witherspoon, \emph{PBW deformations of braided products}, arXiv:1601.02274v2, 2017.




\bibitem{WSL}
X. Wang, Y. Shen and D.-M. Lu, \emph{Artin-Schelter regularity of $R$-smash products}, preprint.

\bibitem{YZ}
A. Yekutieli and J. J. Zhang, {\em Homological transcendence degree}, Proc. Lond. Math. Soc.,  \textbf{93(3)} (2006), 105-137.


\bibitem{ZZ1}
J. J. Zhang and J. Zhang, \emph{Double Ore extension}, J. Pure Appl. Algebra, \textbf{212(12)} (2008), 2668--2690.



\bibitem{ZVZ}
C. Zhu, F. Van Oystaeyen and Y.-H. Zhang, \emph{Nakayama automorphisms of double Ore extensions of Koszul regular algebras}, Mauscripta Math., \textbf{152} (2017), 555--584.
\end{thebibliography}
\end{document}